\documentclass[]{interact}
\usepackage{tikz}
\usetikzlibrary{calc}
\usetikzlibrary{shapes.geometric, arrows}
\usepackage{cite}
\usepackage{hyperref}
\hypersetup{
	colorlinks=true,
	linkcolor=blue,
	urlcolor=blue,
	citecolor=blue
}
\usepackage{epstopdf}
\usepackage[caption=false]{subfig}

\usepackage[longnamesfirst,sort]{natbib}
\bibpunct[, ]{(}{)}{;}{a}{,}{,}
\renewcommand\bibfont{\fontsize{10}{12}\selectfont}

\theoremstyle{plain}
\newtheorem{theorem}{Theorem}[section]

\newtheorem{proposition}[theorem]{Proposition}

\theoremstyle{definition}
\newtheorem{definition}[theorem]{Definition}

\theoremstyle{remark}

\usepackage[percent]{overpic} 

\begin{document}

\articletype{Research Article}

\title{Regional Control Strategies for a Spatiotemporal SQEIAR Epidemic Model: Application to COVID-19}

\author{\name{ Mohammed Elghandouri \textsuperscript{a,b,c}\thanks{CONTACT M. Elghandouri Email: medelghandouri@gmail.com},	Khalil Ezzinbi\textsuperscript{b,c} and Youness Mezzan\textsuperscript{b}}\affil{
		\textsuperscript{a} Centre INRIA de Lyon, CEI-2 56, Boulevard Niels Bohr, 69 603, Villeurbanne, France.\\ 
		\textsuperscript{b} Cadi Ayyad University, Faculty of Sciences Semlalia, B.P: 2390, Marrakesh, Morocco.\\ \textsuperscript{c} IRD-UMMISCO, 32 Av Henri Varagnat, 93 143 Bondy, France.}
}

\maketitle

\begin{abstract}
In this work, we develop a spatial SEIAR-type epidemic model considering a quarantined population (denoted as Q), which we call the SQEIAR model. The dynamics of the SQEIAR model are described by six Partial Differential Equations (PDEs) that represent the changes in the susceptible, quarantined, exposed, asymptomatic, infected, and recovered populations. Our goal is to reduce the number of susceptible, exposed, asymptomatic, and infected individuals while accounting for the environment, which plays a critical role in the spread of epidemics. We then propose a novel strategy for epidemic control, incorporating two key control measures: regional quarantine for the susceptible population and treatment for the infected. This approach serves as an alternative to widespread quarantine, minimizing the economic, social, and other potential impacts. Additionally, we consider the possibility of reinfection among recovered individuals, a common occurrence in many diseases. To demonstrate the practical utility of our results, a numerical example centered on COVID-19 is presented.
\end{abstract}


\begin{amscode}
	35A01; 49J20, 93C10.
\end{amscode}

\section{Introduction} 	\label{section 1}
\noindent

The novel human corona-virus disease 2019 (COVID-19) was first reported in Wuhan, China, in 2019. By September 2021, almost two years after COVID-19 was first identified, there had been more than 200 million confirmed cases and over 4.6 million lives lost to the disease. Moreover, COVID-19 is one of the world's biggest problems due to its economic, social, and political repercussions and crises, see for instance \cite{16,17,18,19,20,21} and the references therein. However, when a disease first appears, and in the absence of effective medications or vaccines, quarantine and isolation strategies with some treatment for the infected individuals are the only ways to control its spread. Quarantine, a restriction placed on the movement of populations, animals, and goods to prevent the spread of diseases, is a natural strategy that has been used for thousands of years. It was used to control the spread of the Black Death in 1793, the Great Influenza in 1918, the Ebola epidemic in 2014, and recently COVID-19 in 2020, (see for instance, \cite{22,23,24,25,26,27}).

In this context, various mathematical models have been proposed to better understand and control the spread of COVID-19, each incorporating different strategies such as quarantine and vaccination. For example, in \cite{29}, a model was introduced that considered both susceptible and symptomatic infectious individuals, addressing positivity, optimal control, and stability analysis. Further, in \cite{30}, a non-autonomous nonlinear deterministic model was proposed to explore COVID-19 control through social distancing, surface cleaning, precautionary measures for exposed individuals, and fumigation of public spaces. Other models, like the SIR model presented by \cite{Libotte}, focused on optimal vaccination strategies to minimize infection rates. Similarly, \cite{Shen} proposed an optimal control model that included prevention, vaccination, and quick screening. In \cite{Elghandouri}, a VS-EIAR epidemiological model was developed to analyze the dynamics of COVID-19, focusing on minimizing susceptible, exposed, infected, and asymptomatic individuals through vaccination and treatment. These and other studies, including \cite{Kohler, Morato, Carli, Peni, Scarabaggio, Aquino, Davies, Guner, Moghadas, Jin, Pellis, Iwami, Okyer, Rachah}, underline the critical role of mathematical modeling in developing effective strategies to combat pandemics.

Building on these foundational works, the authors in \cite{28} recently proposed an impulsive ordinary SQEIAR epidemic model to control the spread of COVID-19. While their model utilized Pontryagin’s Maximum Principle (PMP) to construct an optimal control, it did not account for spatial factors or the possibility of reinfection among recovered individuals-both critical in disease transmission. To address these limitations, we introduce a spatial SEIAR epidemic model that incorporates these missing elements, offering a more comprehensive and realistic approach to epidemic control. Our model not only considers the susceptible, exposed, asymptomatic, infected, and recovered populations but also adds a quarantined group to better capture the dynamics of disease spread. We aim to reduce the numbers of susceptible, exposed, asymptomatic, and infected individuals while increasing the quarantined and recovered populations. To achieve this, we propose two control strategies: quarantine for susceptible individuals and treatment for the infected. Recognizing the high cost of quarantining an entire population, we introduce a new strategy that applies quarantine selectively to specific regions, thereby minimizing the economic, social, and other potential impacts. This strategy also takes into account environmental factors that significantly influence disease transmission, and we call it \emph{Regional Quarantine}. We complement this strategy with treatment measures for infected individuals to save lives and help more people recover from the infection. This spatially aware approach, combined with our focus on reinfection risks, provides a more robust framework for controlling the spread of infectious diseases.

Then, a set of Partial Differential Equation (PDEs) with two control functions is proposed (Equation \eqref{eq 1}). As we will see, due to the presence of the nonlinear term, this equation is not well-defined in the space \(\mathbb{L}^{2}(\Omega)\), where \(\Omega\) represents the domain where the disease emerges. This issue poses a challenge when addressing the existence of solutions for our system. To overcome this problem, we employ the technique of truncation functions. This approach allows us to modify the nonlinear term so that it remains within the bounds of the \(\mathbb{L}^{2}(\Omega)\) space, ensuring that the equation is well-defined. Consequently, we can apply established theorems to prove the existence of a solution. In particular, we rely on the results of the following theorem:

\begin{theorem}\cite[Theorem 1.4]{pazy}
	Let $F:\mathbb{R}^{+}\times X\to X$ be locally integrable w.r.t the first argument and locally Lipschitz continuous w.r.t the second argument. If $B$ is the infinitesimal generator of a $C_0$-semigroup $(T(t))_{t\geq 0}$ on $X$, then for every $y_0\in X$, there exists a $t_{max}\leq +\infty$ such that the following initial value Cauchy problem:
	\begin{equation*}
		\left\{\begin{array}{l}
			y'(t)=B y(t)+F(t,y(t)), \quad t\geq 0\\[5pt]
			y(0)=y_0,
		\end{array}\right.
		\label{Eq 8}
	\end{equation*}
	has a unique (mild) solution $y$ on $[0,t_{max})$. Moreover, if $t_{max}<+\infty$, then $\limsup\limits_{t\to t_{max}}\Vert y(t)\Vert_{X}=+\infty$.
	\label{pazy_thm}
\end{theorem}

The paper is structured as follows: in Section \ref{section 2}, we introduce our model. Section \ref{section 3} is dedicated to the examination of the existence of solutions for the proposed model. Following that, in Section \ref{section 4}, we delve into the exploration of optimal solutions, while Section \ref{section 5} focuses on the study of optimal controls. Section \ref{section 6} demonstrates the application of our theoretical framework to COVID-19 through numerical simulation. The final section offers concluding remarks.

\section{SQEIAR Epidemic Model} \label{section 2}
\noindent

Let $\Omega$ be a connected, bounded domain in $\mathbb{R}^{d}$; $d \in \{1,2,3\}$ with a smooth boundary $\partial\Omega$. Let $x \in \Omega$ and $t \in [0, \tau]$ for $\tau > 0$. The nonlinear SQEIAR epidemiological model involves six non-negative state variables: $S(t,x)$, $Q(t,x)$, $E(t,x)$, $A(t,x)$, $I(t,x)$, and $R(t,x)$. In this context, $S(t,x)$ represents the total number of population at risk to catch the infection at time $t$ and position $x$, i.e.,  individuals who are susceptible to infection but not yet infected. $Q(t,x)$ denotes the number of quarantined individuals at time $t$ and position $x$. When susceptible individuals contract the disease, they enter the exposed group, denoted as $E(t,x)$; they are infected, but we cannot detect them. After a certain period of time, the exposed individuals start transmitting the infection at a rate $k$. Some of them exhibit symptoms, categorized as $I(t,x)$, while the others show no visible symptoms, labeled as $A(t,x)$. After a period of time, a fraction \((1-p)\) of the asymptomatic population moves to the infected group, while the remainder transitions to the group of individuals who have recovered, represented by \(R(t,x)\). A part $(1-\alpha)f$ of infected population will die because of the infection, while the rest joins the recovered group.  Furthermore, some recovered individuals may be susceptible to reinfection at a rate $\xi$.

To apply our control strategy mentioned in paragraph three of the introduction, we assume that quarantine must be implemented for the susceptible population in \(n\) (\(n \in \mathbb{N}^*\)) fixed regions, denoted as \(\omega_i\) (\(i=1, \ldots, n\)), within the total space \(\Omega\). Dividing the total space into smaller regions provides a more practical approach, allowing us to focus on areas where the disease spreads more rapidly due to factors such as population density, mobility, healthcare access, and environmental conditions. Quarantine is applied to the susceptible population in each region \(\omega_i\) at a rate \(\chi_{\omega_i}(x)v(t,x)\), $x\in \Omega$, where \(\chi_{\omega_i}\) is the characteristic function on $\omega_i$ (see Figure \ref{fig0}).  Thus, the total quarantine rate applied across all specified regions within the space $\Omega$ is give by 
\[ \sum_{i=1}^{n}\chi_{\omega_i}(x) v(t,x),\quad x\in \Omega.\]
In addition to quarantine, treatment is applied to infected individuals at a rate \(u(t,x)\) as a complementary control measure to minimize the mortality rate \((1-\alpha)f\), while also reducing the transmission rate. Figure \ref{Fig 1} provides a visual summary of the biological dynamics of our model with these control measures, illustrating the interaction between susceptible, quarantined, exposed, asymptomatic, infected, and recovered individuals.
\begin{figure}[h!]
	\centering
	\begin{tikzpicture}[scale=1.1]
	\draw[thick, fill=blue!10] (3, 2.5) ellipse (3.1cm and 2.1cm);
	\node at (5.5, 2.8) {\(\bf{ \Omega}\)};
	
	\draw[thick, draw=red, fill=red!30] (1.5, 3.5) ellipse (0.7cm and 0.5cm);
	\node at (1.5, 3.5) {\(\bf{\omega_1}\)};
	
	\node[red] at (1.3, 3.75) {$\ast$};
	\node[blue] at (1.85, 3.3) {$\star$};
	\node[blue] at (1.1, 3.5) {$\star$};
	\node[purple] at (1.35, 3.2) {¤};
	\node[cyan] at (1.8, 3.7) {$\divideontimes$};
	
	\draw[thick, green!60!black] (1.5, 3.5) ellipse (0.77cm and 0.57cm);
	
	\draw[thick,draw=red, fill=red!30] (3.5, 3.5) ellipse (0.8cm and 0.5cm);
	\node at (3.5, 3.5) {\(\omega_2\)};
	
	\node[red] at (3.1, 3.7) {$\ast$};
	\node[blue] at (3.8, 3.2) {$\star$};
	\node[red] at (4, 3.5) {$\ast$};
	\node[purple] at (3.4, 3.2) {¤};
	\node[cyan] at (3, 3.4) {$\divideontimes$};
	\node[cyan] at (3.5, 3.8) {$\divideontimes$};
	
	\draw[thick, green!60!black] (3.5, 3.5) ellipse (0.87cm and 0.57cm);
	
	\node at (4.5, 2.5) {\(\bf{\cdots}\)};
	
	\draw[thick, draw=red, fill=red!30] (4.5, 1.5) ellipse (0.8cm and 0.5cm);
	\node at (4.5, 1.5) {\(\bf{\omega_{n-1}}\)};
	
	\node[red] at (4.3, 1.8) {$\ast$};
	\node[red] at (5.1, 1.5) {$\ast$};
	\node[blue] at (4.7, 1.8) {$\star$};
	\node[blue] at (4.8, 1.2) {$\star$};
	\node[purple] at (4.4, 1.2) {¤};
	\node[cyan] at (3.9, 1.5)  {$\divideontimes$};
	
	\draw[thick, green!60!black] (4.5, 1.5) ellipse (0.87cm and 0.57cm);
	
	\draw[thick, draw=red, fill=red!30] (2, 1.5) ellipse (0.6cm and 0.7cm);
	\node at (2, 1.5) {\(\bf{\omega_n}\)};
	
	\node[red] at (1.6, 1.6) {$\ast$};
	\node[blue] at (2.2, 1.15) {$\star$};
	\node[purple] at (1.8, 1.15) {¤};
	\node[purple] at (2.35, 1.7) {¤};
	\node[cyan] at (1.9, 1.9) {$\divideontimes$};
	
	\draw[thick, green!60!black] (2, 1.5) ellipse (0.67cm and 0.77cm);
	
	\node[draw, text width=5.25cm, align=left, thick] at (9, 3.2) {
		\textcolor{red}{$\ast$ : Mobility factors} \\ 
		\textcolor{cyan}{$\divideontimes$ : Healthcare access}\\
		\textcolor{purple}{¤ : Population density} \\
		\textcolor{blue}{$\star$ : Other relevant factors}
	};
\node[draw, text width=5.25cm, align=left, thick] at (9, 1.5) {\textcolor{green!60!black}{\textcircled{} : Quarantine control} at rate 
	\begin{center}
	\bf{\(\chi_{\omega_i}(x) \textcolor{green!60!black}{v(t,x)}, \hspace{0.2cm} x \in \Omega\)}.
	\end{center}
};
	\end{tikzpicture}
	\caption{Quarantine control must be applied in the highlighted sub-regions \(\omega_i\) of \(\Omega\), where key factors influencing disease spread are present.}
	\label{fig0}
\end{figure}
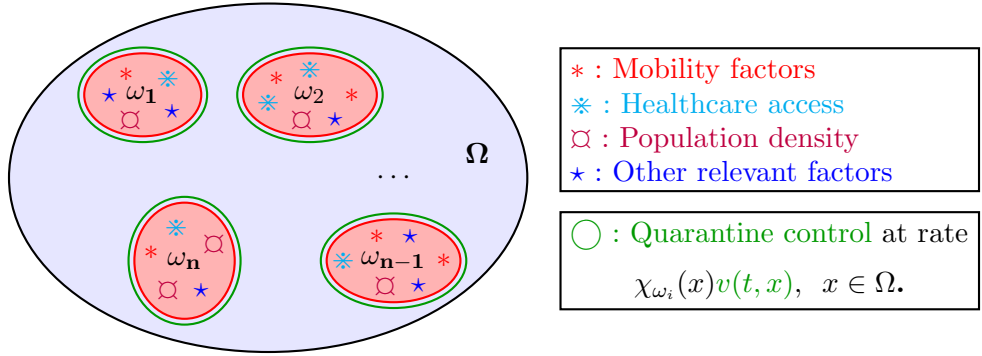

 The dynamics of our model are characterized by the following system of Partial Differential Equations (PDEs):
\begin{equation}
	\begin{cases}
		\begin{aligned}
			\frac{\partial S(t,x)}{\partial t} &= D_1\Delta S(t,x) - \Lambda(t,x)S(t,x) + \xi R(t,x) - \chi_{\Omega_n}(x) v(t,x)S(t,x),\\[5pt]
			\frac{\partial Q(t,x)}{\partial t} &= D_2\Delta Q(t,x) +\chi_{\Omega_n}(x) v(t,x)S(t,x),\\[5pt]
			\frac{\partial E(t,x)}{\partial t} &= D_3\Delta E(t,x) - kE(t,x) + \Lambda(t,x)S(t,x),\\[5pt]
			\frac{\partial A(t,x)}{\partial t} &= D_4\Delta A(t,x) - \eta A(t,x) + (1-z)kE(t,x),\\[5pt]
			\frac{\partial I(t,x)}{\partial t} &= D_5\Delta I(t,x) + zkE(t,x) + (1-p)\eta A(t,x)- (f+u(t,x))I(t,x),\\[5pt]
			\frac{\partial R(t,x)}{\partial t} &= D_6\Delta R(t,x) - \xi R(t,x) + \alpha fI(t,x) + u(t,x)I(t,x) + p\eta A(t,x),
		\end{aligned}
	\end{cases}
	\label{eq 1}
\end{equation}
for $t\in[0,\tau]$ and $x\in\Omega$. Here, $D_i$ for $i=1,\ldots,6$ represent diffusion coefficients, and \( 0 \leq u(t,x) \leq 1 \) and \( 0 \leq v(t,x) \leq \frac{1}{n} \) represent the treatment and quarantine control actions, respectively. The control \( u(t,x) \) corresponds to the intensity of treatment applied to infected individuals, aiming to reduce the severity of the infection and improve recovery rates. On the other hand, \( v(t,x) \) represents the level of quarantine applied to susceptible individuals in a given region. The function $\chi_{\Omega_n}(x)$ is the characteristic function on $\Omega_n$, where $\Omega_n=\bigcup\limits_{i=1}^{n}\omega_i$ such that $\left\{\omega_i\right\}_{i=1}^{n}\subset \Omega$ with $\bigcap\limits_{i=1}^{n}\omega_i=\emptyset$. In this case \(\sum\limits_{i=1}^{n}\chi_{\omega_i}=\chi_{\Omega_n}\).  The term $\Lambda(t,x)=\delta E(t,x)+(1-q)I(t,x)+\mu A(t,x)$, where $\delta>0$, $1-q>0$, and $\mu>0$ represent the reduced transmissibility factors associated with contacts from exposed, infected, and asymptomatic individuals, respectively. The parameters $\eta$, $p$, $k$, $z$, $\alpha$, $f$, and $\xi$ are all in the interval $[0,1]$ and $D_i>0$, $i=1,\ldots,6$.

On the boundary $\partial \Omega$ of $\Omega$, Neumann boundary conditions are employed. This implies that the flow is null on $\partial \Omega$ for all variables:
\begin{equation}
	\dfrac{\partial S(t,x)}{\partial \overrightarrow{\nu}} = \dfrac{\partial Q(t,x)}{\partial \overrightarrow{\nu}} = \dfrac{\partial E(t,x)}{\partial \overrightarrow{\nu}} = \dfrac{\partial A(t,x)}{\partial \overrightarrow{\nu}} = \dfrac{\partial I(t,x)}{\partial \overrightarrow{\nu}} = \dfrac{\partial R(t,x)}{\partial \overrightarrow{\nu}} = 0,
	\label{equ 2.2}
\end{equation}
for $t \in [0,\tau]$ and $x \in \partial\Omega$, where $\overrightarrow{\nu}$ is the unit normal vector on $\partial\Omega$. For the initial conditions, we denote by
\begin{equation}\label{equ 2.3}
	\begin{aligned}
		&\left(S(0,x), Q(0,x), E(0,x), A(0,x), I(0,x), R(0,x)\right) \\[5pt]
		&= \left(S_0(x), Q_0(x), E_0(x), A_0(x), I_0(x), R_0(x)\right),
	\end{aligned}
\end{equation}
the initial susceptible, quarantined, exposed, asymptomatic, infected, and recovered individuals, respectively, at position $x \in \Omega$. Let $N(t,x)=S(t,x)+E(t,x)+A(t,x)+I(t,x)+R(t,x)+Q(t,x)$ be the total population at time $t$ and at position $x\in \Omega$. We denote by $\Omega_\tau:=[0,\tau]\times\Omega$.
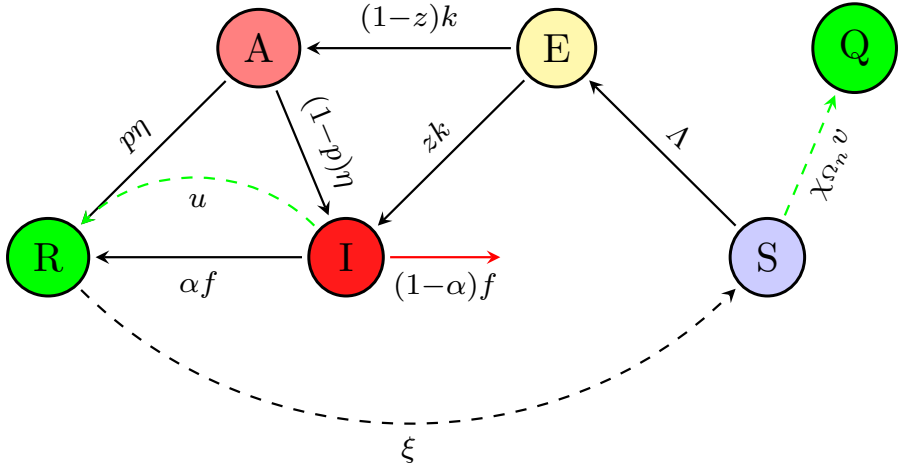
\begin{figure}[ht]
	\centering
	\scalebox{1.4}{
	\begin{tikzpicture}[node distance=2.8cm, auto, thick, >=stealth]
	
	\node [ellipse, draw, fill=yellow!40, align=center, minimum width=0.7cm, minimum height=0.7cm] (E) {E};
	\node [ellipse, draw, fill=red!50, left of=E, align=center, minimum width=0.7cm, minimum height=0.7cm] (A) {A};
	\node [ellipse, draw, fill=red!90, below left of=E, align=center, minimum width=0.7cm, minimum height=0.7cm] (I) {I};
	\node [ellipse, draw, fill=green, below left of=A, align=center, minimum width=0.7cm, minimum height=0.7cm] (R) {R};
	\node [ellipse, draw, fill=green, right of=E, align=center, minimum width=0.7cm, minimum height=0.7cm] (Q) {Q};
	\node [ellipse, draw, fill=blue!20, below right of=E, align=center, minimum width=0.7cm, minimum height=0.7cm] (S) {S};
	
	\draw[->, line width=0.6pt, >=stealth, shorten <=0.5mm, shorten >=0.5mm] (S) -- (E) node[midway, sloped, above] {\(\scriptstyle \Lambda\)};
	\draw[ dashed,->, draw=green, line width=0.6pt, shorten <=0.5mm, shorten >=0.5mm] (S) -- (Q) node[midway, sloped, below] {\(\scriptstyle\chi_{\Omega_n}v\)};
	\draw[->, line width=0.6pt, shorten <=0.5mm, shorten >=0.5mm] (E) -- node[midway, sloped, above] {\(\scriptstyle (1-z)k \)} (A);
	\draw[->, line width=0.6pt, shorten <=0.5mm, shorten >=0.5mm] (A) -- node[midway, sloped, above] {\(\scriptstyle (1-p)\eta \)} (I);
	\draw[->, line width=0.6pt, shorten <=0.5mm, shorten >=0.5mm] (A) -- node[midway, sloped, above] {\(\scriptstyle p\eta \)} (R);
	\draw[->, line width=0.6pt, shorten <=0.5mm, shorten >=0.5mm] (I) -- node[midway, sloped, below] {\(\scriptstyle \alpha f \)} (R);
	\draw[dashed,->, draw=green, line width=0.6pt, shorten <=0.5mm, shorten >=0.5mm, bend right=45] (I) to node[midway, sloped, below] {\(\scriptstyle u\)} (R);
	\draw[->, line width=0.6pt, shorten <=0.5mm, shorten >=0.5mm] (E) -- node[midway, sloped, above] {\(\scriptstyle zk \)} (I);
	\draw[->, line width=0.6pt, shorten <=0.5mm, shorten >=0.5mm, draw=red] (I) -- ++(1.5,0) node[midway, sloped, below] {\(\scriptstyle (1-\alpha) f\)};
	\draw[dashed,->, line width=0.6pt, shorten <=0.5mm, shorten >=0.5mm, bend right=45] (R) to node[midway, sloped, below] {\(\scriptstyle \xi  \)} (S);
	
	\end{tikzpicture}
}
	\caption{Biological Processes of the SQEIAR Model with our Proposed Controllers.}
	\label{Fig 1}
\end{figure}
\vspace{-1cm}
\section{Existence of solutions} \label{section 3}
\noindent

In this section, we investigate the existence of both strong and mild solutions for the system \eqref{eq 1}-\eqref{equ 2.3}. Additionally, we establish the positiveness and boundedness of these solutions. Let $y=(y_1,\ldots,y_6)=(S,Q,E,A,I,R)$ and $y_0=(y_{1,0},\ldots,y_{1,6})=(S_0,Q_0,E_0,A_0,I_0,R_0)$ such that $y(t,x):=y(t)(x)\break=(y_1(t)(x),\ldots,y_6(t)(x))=(y_1(t,x),\ldots,y_6(t,x))\!=\!(S(t,x),Q(t,x),E(t,x),\break A(t,x),I(t,x),R(t,x))$ for $(t,x)\in\Omega_\tau$. Let $H=\left(\mathbb{L}^{2}(\Omega)\right)^{6}$ be the Hilbert space endowed with the norm $\Vert\cdot\Vert$ defined by 
\begin{center}
	$\Vert z\Vert=\left(\sum\limits_{i=1}^{6}\displaystyle\int_{\Omega}\vert z_i(x)\vert^2dx\right)^{\frac{1}{2}}$,
\end{center}
for $z=(z_1,\ldots,z_6)\in H$. We define the linear operator $(B,D(B))$; $D(B)$ is the domain of $B$, by 
\begin{equation}
	\left.\begin{array}{l}
		Bz=D\Delta z \quad \text{for} \quad z\in D(B)=\left( H_0^{2}(\Omega)\right)^{6},
	\end{array}\right.
	\label{A}
\end{equation}
where $D$ is the diagonal matrix given by
\[D=\text{diag}\left(\left\{D_i\right\}_{1\leq i\leq 6}\right),\]
and
\begin{center}
	$H^{2}_0(\Omega)=\left\{z\in H^{2}(\Omega)\mid \text{ such that } \dfrac{\partial z}{\partial \overrightarrow{\nu}}=0 \text{ for a.e. on }\partial\Omega\right\}$.
\end{center}
Let $U=\left(\mathbb{L}^{2}(\Omega)\right)^2$. We define the admissible control set as:
\[\mathcal{U}_{ad} = \left\{\begin{array}{l} (u,v)\in U \mid 0\leq u(t,x)\leq 1 
	\text{ and } 0\leq v(t,x)\leq \dfrac{1}{n} \text{ for } (t,x)\in\Omega_\tau \end{array}\right\}.\]
Let $F: [0, \tau] \times D(F) \to H$ be the function defined as $F(t, z) = (F_1(t, z), \ldots, F_6(t, z))$, where
\begin{equation*}
	\left.\begin{array}{l}
		F_1(t, z)(x) = -\left[\delta z_3(x)\! +\! (1-q)z_5(x)\! +\! \mu z_4(x)\right] z_1(x)\! +\! \xi z_6(x)\!-\!\chi_{\Omega_n}(x) v(t, x)z_1(x),\\[10pt]
		F_2(t, z)(x) = \chi_{\Omega_n}(x) v(t, x)z_1(x),\\[10pt]
		F_3(t, z)(x) = -kz_3(x) + \left[\delta z_3(x) + (1-q)z_5(x) + \mu z_4(x)\right] z_1(x),\\[10pt]
		F_4(t, z)(x) = -\eta z_4(x) + (1-z)kz_3(x),\\[10pt]
		F_5(t, z)(x) = -fz_5(x) + zkz_3(x) + (1-p)\eta z_4(x) - u(t, x)z_5(x),\\[10pt]
		F_6(t, z)(x) = -\xi z_6(x) + \alpha fz_5(x) + u(t, x)z_5(x) + p\eta z_4(x),
	\end{array}\right.
\end{equation*}
for \(t \in [0, \tau]\), \(z \in D(F) = \left\{z = (z_1, \ldots, z_6) \in H \mid F(t, z) \in H \text{ for all } t \in [0, \tau] \right\}\), and fixed \(u, v \in \mathcal{U}_{ad}\). Under the above notations, system \eqref{eq 1}-\eqref{equ 2.3} can be written in the following abstract form: 
\begin{equation}
	\left\{\begin{array}{l}
		y'(t)=By(t)+F(t,y(t)), \quad t\in[0,\tau]\\[5pt]
		y(0)=y_0.
	\end{array}\right.
	\label{equation 2}
\end{equation}
We denote by $W^{1,2}\left(0,\tau;\mathbb{L}^{2}(\Omega)\right)$ the space of absolutely continuous functions $y:[0,\tau]\to \mathbb{L}^{2}(\Omega)$ such that $\dfrac{\partial y}{\partial t}\in \mathbb{L}^{2}(0,\tau;\mathbb{L}^{2}(\Omega))$.
\begin{proposition}
	$B$ is an infinitesimal generator of a $C_0$-semigroup $(T(t))_{t\geq 0}$  of contraction on $H$. Moreover, $T(t)$ is compact for $t>0$.
\end{proposition}
\begin{proof} It suffices to show that $(-B,D(B))$ is maximal monotone on $H$. This follows from the fact that for each $D_i>0$; $i=1\ldots,6$, the operator $(-D_i\Delta,H_0^2(\Omega))$ is maximal monotone on $\mathbb{L}^{2}(\Omega)$. The compactness follows the fact that for each $D_i>0$; $i=1\ldots,6$, the operator $(D_i\Delta,H_0^2(\Omega))$ generates a compact semigroup on $\mathbb{L}^{2}(\Omega)$. For more details, we refer to \cite{pazy}.
\end{proof}
Let $H^{+}:=\{z=(z_1,\ldots,z_6)\in H\mid \text{ such that } z_i\geq 0 \text{ for } i=1,\ldots,6\}$. The following theorem demonstrates the existence of strong solutions for equation \eqref{equation 2}.
\begin{theorem}
	Let $(u, v)\in \mathcal{U}_{ad}$ be fixed. If $y_0\in D(B)\cap H^{+}$, then equation (\ref{equation 2}) has a unique strong solution
	\[ y = (y_1, \ldots, y_6) \!\in\! \left(W^{1,2}\left(0, \tau; \mathbb{L}^{2}(\Omega)\right) \cap \mathbb{L}^{2}(0, \tau; H_0^{2}(\Omega)) \cap \mathbb{L}^{\infty}(0, \tau; H^{1}(\Omega)) \right)^{6}.\]
	Moreover, $y_1, \ldots, y_6$ are positives and uniformly bounded in $\mathbb{L}^{\infty}(\Omega_\tau)$. There exists $C > 0$, independent of $(u, v)$, such that for $i = 1, \ldots, 6$,
	\begin{equation}
		\left\Vert \dfrac{\partial y_i}{\partial t}\right\Vert_{\mathbb{L}^{2}(\Omega_\tau)} + \Vert y_i\Vert_{\mathbb{L}^{2}(0, \tau; H_0^{2}(\Omega))} + \Vert y_i\Vert_{\mathbb{L}^{\infty}(0, \tau; H^{1}(\Omega))} + \Vert y_i\Vert_{\mathbb{L}^{\infty}(\Omega_\tau)} \leq C.
		\label{3.3}
	\end{equation}
\end{theorem}
\begin{proof} The function \( F \) is not defined on the entire space \( H \), which means that equation \eqref{equation 2} is not well-defined on the space \( H \). Then, we cannot apply directly Theorem \ref{pazy_thm}. Therefore, the truncation technique is applied. Specifically, the truncation procedure is applied to function \( F \). For simplicity, consider a function \( G \) and any positive integer \( m \).
	Define the sets:
	\begin{equation*}
		\left\{\begin{array}{l}
			D_1^m(G) = \{(t, z) \mid -m \leq G(t, z) \leq m \}\\[10pt]
			D_2^m(G) = \{(t, z) \mid G(t, z) < -m\}\\[10pt]
			D_3^m(G) = \{(t, z) \mid G(t, z) > m\}.
		\end{array}\right.
	\end{equation*}
	Further, define the truncation form of the function $G(t, z)$, denoted by $G^m(t, z)$, as
	\begin{equation*}
		G^m(t, z) = \left\{\begin{array}{l}
			G(t, z) \quad \text{for } (t, z) \in D_1^m(G)\\[10pt]
			- m \quad \text{for } (t, z) \in D_2^m(G)\\[10pt]
			m \quad \text{for } (t, z) \in D_3^m(G).
		\end{array}\right.
	\end{equation*}
	Then, the truncation form of $F$ is defined as $F^m(t,z)=\left(F_1^m(t,z),\ldots,F_s^m(t,z)\right)$ for $(t,z)\in \bigcup\limits_{l=1}^{3}D_l^m(F)$. It is easily seen that \( F^m(t,z) \) is bounded in \( t \in [0, \tau] \) and is locally Lipschitz continuous in \( z \in H \), uniformly with respect to \( t \in [0, \tau] \). By Theorem \ref{pazy_thm}, we show that equation:
	\begin{equation}\label{equation 3.4}
		\left\{\begin{array}{l}
			y'(t)=By(t)+F^m(t,y(t)), \quad t\in[0,\tau]\\[5pt]
			y(0)=y_0 \in D(B),
		\end{array}\right.
	\end{equation}
	has a unique strong solution $y^m=(y^m_1,\ldots,y^m_6)$ defined on $[0,\tau_m]$ for some $0<\tau_m\leq \tau$ and satisfying
	
	\begin{equation}
		y^m\in \left(W^{1,2}\left(0,\tau_m;\mathbb{L}^{2}(\Omega)\right)\cap \mathbb{L}^{2}(0,\tau_m;H_0^{2}(\Omega)) \right)^{6}.
	\end{equation}
	Let $M>0$ be large enough such that $F^m(t,z)+Mz\geq 0$ for each $z\in H^{+}$ and $t\in [0,\tau_m]$. Equation \eqref{equation 3.4} can be written in the following equivalent form:
	
	\begin{equation}
		\left\{
		\begin{array}{l}
			y'(t) = (B - MI_d)y(t) + F^m(t, y(t)) + My(t), \quad t\in[0,\tau]\\[10pt]
			y(0) = y_0,
		\end{array}
		\right.
		\label{equation 3}
	\end{equation}
	where $I_d$ is the identity map on $H$. The strong solution of equation \eqref{equation 3} is given by 
	\begin{equation*}
		y^m(t) = e^{-Mt}T(t)y_0 + \int_{0}^{t}\!\!e^{-M(t-s)}T(t-s)\left[F^m(s,y^m(s))+My^m(s)\right] \, ds,
	\end{equation*}
	for $t\in[0,\tau_m]$. Since $y_0\geq 0$ and $F^m(t,z)+Mz\geq 0$ for $z\in H^{+}$ and $t\in [0,\tau_m]$, it follows that $y^m(t,x)\geq 0$, that is $y^m_i(t,x)\geq 0$ for $(t,x)\in [0,\tau_m]\times\Omega$ and $i=1,\ldots,6$. Next, we show that $y^m_1,\ldots,y^m_6$ are uniformly bounded in $\mathbb{L}^{\infty}([0,\tau_m]\times\Omega)$. Let $(T_i(t))_{t\geq 0}$, for $i=1,\ldots,6$ be the $C_0$-semigroups  generated by $(B_i,D(B_i))$ defined by 
	\begin{center}
		$B_iz=D_i\Delta z$ \text{ for } $z\in D(B_i)=H^{2}_0(\Omega)$.
	\end{center}
	Let $M^m_i=\max\left\{\Vert F^m_i(\cdot,y^m(\cdot))\Vert_{\mathbb{L}^{\infty}(\Omega_\tau) }, \Vert y_{i,0}\Vert_{\mathbb{L}^{\infty}(\Omega)} \right\}$. It is obvious that function $z^m_i(t,x)=y^m_i(t,x)-M^m_it-\Vert y_{i,0}\Vert_{\mathbb{L}^{\infty}(\Omega)}$ satisfies the following equation:
	
	\begin{equation}
		\left\{	\begin{array}{l}
			\dfrac{\partial z^m_i(t,x)}{\partial t}=B_iz^m_i(t,x)+F^m_i(t,y^m(t))(x)-M^m_i\\[10pt]
			z^m_i(0,x)=y_{i,0}-\Vert y_{i,0}\Vert_{\mathbb{L}^{\infty}(\Omega)}.
		\end{array}\right.
		\label{equation 4}
	\end{equation}
	for $(t,x)\in [0,\tau_m]\times\Omega$. Equation (\ref{equation 4}) has a unique strong solution defined by 
	\begin{eqnarray*}
		z^m_i(t,x)&=&(T_i(t)(y_{i,0}-\Vert y_{i,0}\Vert_{\mathbb{L}^{\infty}(\Omega)}))(x)\\
		& &+\displaystyle\int_{0}^{t} \left(T_i(t-s)\left[F^m_i(s,y^m(s))-M^m_i\right]\right)(x)ds,
	\end{eqnarray*}
	for $(t,x)\in [0,\tau_m]\times\Omega$. Since $y_{i,0} - \Vert y_{i,0}\Vert_{\mathbb{L}^{\infty}(\Omega)} \leq 0$, and $F^m_i(t,y^m(t)) - M^m_i \leq 0$ for $t \in [0,\tau_m]$, it follows that $z^m_i(t,x) \leq 0$ for each $(t,x) \in [0,\tau_m]\times\Omega$. Then,
	\begin{equation}\label{M_i}
		0 \leq y^m_i(t,x) \leq M^m_it + \Vert y_{i,0}\Vert_{\mathbb{L}^{\infty}(\Omega)}, \quad i=1,\ldots,6, \quad (t,x) \in [0,\tau_m]\times\Omega.
	\end{equation}
	As a consequence, $y^m_i \in \mathbb{L}^{\infty}([0,\tau_m]\times\Omega) \hspace{0.1cm} \text{ for } \hspace{0.1cm} i=1,\ldots,6$, moreover,
	\[
		0\leq  y^m_i(t,x)\leq M^m_i\tau_m+\Vert y_{i,0}\Vert_{\mathbb{L}^{\infty}(\Omega)}, \quad i=1,\ldots,6, \quad  (t,x) \in \Omega_\tau.
	\]
	This implies also that $\tau_m=\tau$. 
	
	As $M^m_i$ is bounded, there exists a constant $t_{max} \in (0, \tau)$ such that we can choose an integer $m > 0$ satisfying $m > M^m_i t_{max} + \| y_{i,0} \|_{\mathbb{L}^{\infty}(\Omega)}$.
	 By \eqref{M_i}, we get that
	\begin{center}
		$0 \leq y^m_i(t,x)\leq m$, \quad $i=1,\ldots,6$, \quad  $(t,x) \in [0,t_{max})\times \Omega$.
	\end{center}
	Therefore, from the definition of \( F^m(t, y^m(t)) \), it follows that \( F^m(t, y^m(t)) = F(t, y(t)) \) and \( y^m(t) = y(t) \) for \( t \in [0, t_{max})\). Thus, \( y \) is a local solution of equation \eqref{equation 2} on \( [0, t_{max}) \times \Omega \). Next, we show that the local positive solution \( y \) is, in fact, a global solution in \( \Omega_\tau \). To establish this, it suffices to demonstrate that \( y_i \), for \( i = 1, \ldots, 6 \), are uniformly bounded with respect to \( t_{max} \) in $[0,t_{max})\times \Omega$. Let 	
	\[
	N(t) = \int_{\Omega} N(t, x) \, dx \quad \text{for} \quad t \in [0, t_{max}),
	\]
	be the total population in the entire spatial domain \( \Omega \) at time \( t \in [0, t_{max}) \). As \( y \) is a strong solution of equation \eqref{equation 2}, we obtain
	\[
	\frac{\partial N(t)}{\partial t} = \int_{\Omega} \frac{\partial N(t,x)}{\partial t} \, dx = \int_{\Omega} \left( \sum_{i=1}^{6} D_i \Delta y_i(t,x) \right) \, dx - (1-\alpha) f y_5(t,x).
	\]
	Since \( \dfrac{\partial y_i(t,x)}{\partial \overrightarrow{\nu}} = 0 \) on \( \partial \Omega \), using the Green Formula, we obtain
	\[
	\frac{\partial N(t)}{\partial t} = \int_{\Omega} \frac{\partial N(t,x)}{\partial t} \, dx = - (1-\alpha) f y_5(t,x).
	\]
	As \( y_5(t,x) \geq 0 \) for \( (t,x) \in [0, t_{max}) \times \Omega \) and \( \alpha,f \in [0,1] \), it follows that
	$\frac{\partial N(t)}{\partial t} \leq 0$.
	This implies that \( N(t) \leq N(0) \) for all \( t \in [0, t_{max}) \). Since \( N(t) \) represents the total population in the spatial domain \( \Omega \) at time \( t \in [0, t_{max}) \), we have
	\[
	N(t,x) \leq N(t) \leq N(0) \quad \text{for} \quad (t,x) \in [0, t_{max}) \times \Omega.
	\]
	Due to the positivity of the solution, we obtain
	\[
	0 \leq y_i(t,x) \leq N(0), \quad i = 1, \ldots, 6, \quad (t,x) \in [0, t_{max}) \times \Omega,
	\]
	where
	\[
	N(0) = \int_{\Omega} N(0,x) \, dx.
	\]
	The constant \( N(0) \) is independent of \( t_{max} \). Thus, the solutions \( y_i(t) \), for \( i = 1, \ldots, 6 \), are uniformly bounded with respect to \( t_{max} \) in $[0,t_{max})\times \Omega$. This shows that \( y \) is a global solution of equation \eqref{equation 2} on \( \Omega_\tau \).

	In the sequel, we show that $y=(y_1,\ldots,y_6)\in \left( \mathbb{L}^{\infty}(0,\tau;H^{1}(\Omega)) \right)^{6}$. It follows from system (\ref{eq 1}) that 
	\begin{equation*}
		\dfrac{\partial y_i(t)}{\partial t}=D_i\Delta y_i(t)+F_i(t,y(t)) \hspace{0.1cm} \text{ for } \hspace{0.1cm} i=1,\ldots,6.
	\end{equation*}
	Then, 
	\begin{equation*}
		y_i(t)\dfrac{\partial y_i(t)}{\partial t}=D_iy_i(t)\Delta y_i(t)+y_i(t)F_i(t,y(t)) \hspace{0.1cm} \text{ for } \hspace{0.1cm} i=1,\ldots,6.
	\end{equation*}
	Since $y_i$ is regular, using the Green Formula, we can affirm that
	\begin{eqnarray*}
		\displaystyle\int_{\Omega}y_i(t)\dfrac{\partial y_i(t)}{\partial t}&=&D_i\displaystyle\int_{\Omega} y_i(t)\Delta y_i(t)+\displaystyle\int_{\Omega}y_i(t)F_i(t,y(t))\\
		&=& -D_i \displaystyle\int_{\Omega} \vert\nabla y_i(t)\vert^{2}+\displaystyle\int_{\Omega}y_i(t)F_i(t,y(t)).
	\end{eqnarray*}
	That is,
	\[D_i \displaystyle\int_{\Omega} \vert\nabla y_i(t)\vert^{2}=\displaystyle\int_{\Omega}y_i(t)F_i(t,y(t))-\displaystyle\int_{\Omega}y_i(t)\dfrac{\partial y_i(t)}{\partial t}.\]
	Hence, 
	\begin{eqnarray*}
		 D_i \displaystyle\int_{\Omega} \vert\nabla y_i(t)\vert^{2}\!\!\!\!
		&\leq& \!\!\!\displaystyle\int_{\Omega}\vert y_i(t)\dfrac{\partial y_i(t)}{\partial t}\vert+\displaystyle\int_{\Omega}\vert y_i(t)F_i(t,y(t))\vert\\
		&\leq & \!\!\!\! \left(\displaystyle\int_{\Omega}\vert y_i(t)\vert^2\right)^{\frac{1}{2}}\!\cdot\!\left(\displaystyle\int_{\Omega}\vert \dfrac{\partial y_i(t)}{\partial t}\vert^2\right)^{\frac{1}{2}}\!\!+\!\text{mes}(\Omega)\left(\Vert y_i\Vert_{\mathbb{L}^{\infty}(\Omega_\tau)}\right)\!\cdot\!\left(\Vert F_i\Vert_{\mathbb{L}^{\infty}(\Omega_\tau)}\right)\!\!,
	\end{eqnarray*}
where $\text{mes}(\Omega)$ denotes the measure of $\Omega$. Since $y_i\in W^{1,2}([0,\tau];\mathbb{L}^{2}(\Omega))\cap \mathbb{L}^{\infty}(\Omega_\tau)$, it follows that $y_i\in \mathbb{L}^{\infty}([0,\tau],H^{1}(\Omega))$. By utilizing the facts that $0\leq u(t,x)\leq 1$ and $0\leq v(t,x)\leq \frac{1}{n}$, we can affirm the existence of a constant $C>0$, independent of $u$ and $v$, such that the inequality in (\ref{3.3}) holds for $i=1,\ldots,6$.
\end{proof}
\vspace{-0.8cm}
\section{Existence of the optimal solutions} \label{section 4}
\noindent

In this section, we establish the existence of optimal solutions for system \eqref{eq 1}-\eqref{equ 2.3}. We define the admissible class for system \eqref{eq 1}-\eqref{equ 2.3} by:
\begin{center}
	$\mathcal{A}_{\text{ad}} = \left\{\begin{array}{l} (y, u, v) \mid y\text{ the unique solution of system} \\[5pt] \text{ \eqref{eq 1}-\eqref{equ 2.3} corresponding to } (u, v) \in \mathcal{U}_{a, d} \end{array}\right\}$.
\end{center}
We define the cost functional $\mathcal{J}(\cdot,\cdot,\cdot)$ on $\mathcal{A}_{ad}$ by
\begin{align*}
	\mathcal{J}(y,u,v) &= \int_{0}^{\tau}\int_{\Omega}\left(\rho_1\chi_{\Omega_n}(x)y_1(t,x)  + \sum_{i=3}^{5} \rho_iy_i(t,x)\right)\,dx\,dt \\[8pt]
	&\quad + \dfrac{\sigma_1}{2}\int_{0}^{\tau}\int_{\Omega} u(t,x)^2\,dx\,dt + \dfrac{\sigma_2}{2}\int_{0}^{\tau}\int_{\Omega}\chi_{\Omega_n}(x) v(t,x)^2\,dx\,dt,
\end{align*}
for $(y,u,v)\in \mathcal{A}_{ad}$, where $\rho_1$, $\rho_3$, $\rho_4$, $\rho_5$, $\sigma_1$, and $\sigma_2$ are weight parameters. The term
\begin{center}
	$\displaystyle\int_{\Omega}\left(\rho_1\chi_{\Omega_n}(x)y_1(t,x)  + \sum\limits_{i=3}^{5} \rho_iy_i(t,x)\right)\,dx$,
\end{center}
represents the \textit{epidemic cost} at time $t$. It is a combination of weighted contributions from individuals who are susceptible, exposed, asymptomatic, or infected. Each group has a corresponding weight $\rho_i$ $(i=1,3,4,5)$, which represents the relative importance or severity of the epidemic state. This term assesses the overall impact of the epidemic on the population over time. The expression $\frac{\sigma_1}{2}\int_{\Omega} u(t,x)^2\,dx$ denotes the \textit{control cost} associated with the treatment intervention at time $t$. The parameter $\sigma_1$ serves as the gain of the treatment controller, balancing the significance of implementing the treatment against its associated costs. The term $\frac{\sigma_2}{2}\int_{\Omega}\chi_{\Omega_n}(x) v(t,x)^2\,dx$ represents the \textit{control cost} associated with the quarantined individuals, and it quantifies the cost of implementing quarantine measures at time $t$. The parameter $\sigma_2$ represents the gain of the quarantine controller. The optimal control problem is formulated as:
\begin{equation}
	\min\{\mathcal{J}(y,u,v) \mid (y,u,v)\in\mathcal{A}_{ad}\}.
	\label{eq 4.1}
\end{equation}

\begin{definition}
	A solution of problem \eqref{eq 4.1} is called an optimal solution of \eqref{eq 1}-\eqref{equ 2.3}.
\end{definition}
\begin{definition} Let $(y^*, u^*, v^*)$ be an optimal solution of problem \eqref{eq 4.1}. The control functions $(u^*, v^*)$ are called optimal controls for system \eqref{eq 1}-\eqref{equ 2.3}.
\end{definition}
\begin{theorem} There exist a unique $(y^*,u^*,v^*)\in \mathcal{A}_{ad}$ such that $\mathcal{J}(y^*,u^*,v^*)=\min\{\mathcal{J}(y,u,v)\mid \text{for }(y,u,v)\in \mathcal{A}_{ad} \}$.
\end{theorem}
\begin{proof} Let $\mathcal{I}=\min\{\mathcal{J}(y,u,v)\mid (y,u,v)\in\mathcal{A}_{ad}\}$. Since $0\leq \mathcal{I} <+\infty$, we can affirm that there exists a sequence $\{(y_m,u_m,v_m)\}_{m\geq 0}\subset \mathcal{A}_{a,d}$ such that $\lim\limits_{m\to +\infty}\mathcal{J}(y_m,u_m,v_m)=\mathcal{I}$. Note that, 
	\begin{equation}
		\left\{\begin{array}{l}
			\dfrac{\partial y_m(t)}{\partial t}=By_m(t)+F(t,y_m(t)), \quad t\in[0,\tau]\\[8pt]
			y_m(0)=y_{0}.
		\end{array}\right.
		\label{equation 4.2}
	\end{equation}
	By formula \eqref{3.3}, we show the existence of a constant $\bar{C} > 0$ such that 
	\begin{equation}
		\begin{aligned}
			&\max\left\{\left\|\frac{\partial y_m}{\partial t}\right\|_{\left(\mathbb{L}^{2}(\Omega_\tau)\right)^{6}},\|y_m\|_{\left(\mathbb{L}^{2}(0, \tau; H_0^{2}(\Omega))\right)^{6}},\|y_m\|_{\left(\mathbb{L}^{\infty}(0, \tau; H^{1}(\Omega))\right)^{6}},\|y_m\|_{\left(\mathbb{L}^{\infty}(\Omega_\tau)\right)^{6}}\right\}\! \leq \! \bar{C}.
		\end{aligned}
		\label{4.3}
	\end{equation}
	Thus, $y_m(t)$ is uniformly bounded with respect to $m$. Since $\left(H^{1}(\Omega)\right)^{6}$ is compactly embedded in $H$, we conclude that $\{y_m(t) \mid m \geq 0\}$ is relatively compact in $H$. We also demonstrate that $\{y_m(t) \mid m \geq 0\}$ is equicontinuous in $\mathcal{C}([0,\tau];H)$. Indeed, let $0<t_1<t_2\leq \tau$, then 
	\begin{eqnarray*}
		\Vert y_m(t_2)-y_m(t_1)\Vert &\leq &  \Vert T(t_2)-T(t_1)\Vert_{\mathcal{L}(H)}\Vert y_{0}\Vert\\
		& & + \hspace{0.1cm} \Vert\displaystyle\int_{0}^{t_2}T(t_2-s)F(s,y_m(s))ds-\displaystyle\int_{0}^{t_1}T(t_1-s)F(s,y_m(s))ds\Vert\\
		&\leq & \Vert T(t_2)-T(t_1)\Vert_{\mathcal{L}(H)}\Vert y_{0}\Vert +  \displaystyle\int_{t_1}^{t_2}\Vert T(t_2-s)F(s,y_m(s))\Vert ds\\
		& & + \hspace{0.1cm}
		\displaystyle\int_{0}^{t_1}\Vert T(t_2-s)-T(t_1-s)\Vert_{\mathcal{L}(H)}\Vert F(s,y_m(s))\Vert ds\\
		&\leq& 
		\Vert T(t_2)-T(t_1)\Vert_{\mathcal{L}(H)}\Vert y_{0}\Vert +  \displaystyle\int_{t_1}^{t_2}\Vert F(s,y_m(s))\Vert ds\\
		& & + \hspace{0.1cm}
		\displaystyle\int_{0}^{t_1}\Vert T(t_2-s)-T(t_1-s)\Vert_{\mathcal{L}(H)}\Vert F(s,y_m(s))\Vert ds.
	\end{eqnarray*}
	Since the compactness of $T(t)$ for $t>0$, it follows that $T(t)$ is norm continuous for $t>0$. By the Dominated Convergence Theorem, we get that $\lim\limits_{t_2\to t_1}\Vert y_m(t_2)-y_m(t_1)\Vert=0$ uniformly for $m\in \mathbb{N}$. For $t_1=0$, we have 
	\begin{eqnarray*}
		\Vert y_m(t_2)-y_m(0)\Vert &\leq &  \Vert T(t_2)y_0-y_0\Vert +  \Vert\displaystyle\int_{0}^{t_2}T(t_2-s)F(s,y_m(s))ds\Vert\\
		&\leq & \Vert T(t_2)y_0-y_0\Vert +  \displaystyle\int_{0}^{t_2}\Vert T(t_2-s)F(s,y_m(s))\Vert ds\\
		&\leq& \Vert T(t_2)y_0-y_0\Vert + \displaystyle\int_{0}^{t_2}\Vert F(s,y_m(s))\Vert ds,
	\end{eqnarray*}
	which implies that $\lim\limits_{t_2\to 0}\Vert y_m(t_2)-y_m(0)\Vert=0$ uniformly for $m\in \mathbb{N}$.
	By  Arzelà-Ascoli’s Theorem, we show that $\left\{y_m\mid m\geq 0\right\}$ is compact in $\mathcal{C}([0,\tau];H)$. Hence, there exists a subsequence of $(y_m)_{m\geq 0}$ that we continue to denote by the same index $m\geq 0$ such that $y_m\to y^{*}$ in $\mathcal{C}([0,\tau];H)$ uniformly with respect to $t$. Since the sequence $(By_m)_{m\geq 0}$ is bounded in $H$ (by formula \eqref{4.3}), it follows that there exists a subsequence, also denoted by $(By_m)_{m\geq 0}$, that converges weakly in $\left(\mathbb{L}^{2}(\Omega_\tau)\right)^{6}$. Moreover, for all distributions $\varphi$, we have
	\begin{center}
		$\displaystyle\int_{\Omega_\tau}\varphi By_m = \displaystyle\int_{\Omega_\tau}y_m B\varphi \to \displaystyle\int_{\Omega_\tau}y^{*}B \varphi = \displaystyle\int_{\Omega_\tau}\varphi By^{*}$,
	\end{center}
	which implies that $By_m \rightharpoonup By^*$ weakly in $\left(\mathbb{L}^{2}(\Omega_\tau)\right)^{6}$. From (\ref{4.3}), we get that $\dfrac{\partial y_m}{\partial t}\rightharpoonup \dfrac{\partial y^{*}}{\partial t}$ weakly in $\left(\mathbb{L}^{2}(\Omega_\tau)\right)^{6}$, $y_m\to y^{*}$ weakly in $\left(\mathbb{L}^{2}( 0,\tau; H_0^{2}(\Omega))\right)^{6}$ and $y_m\to y^{*}$ weakly in $\left(\mathbb{L}^{\infty}(0,\tau;H^{1}(\Omega))\right)^{6}$ as $m\to +\infty$. Since the reflexivity of $\mathbb{L}^{2}(\Omega_\tau)$ and $\mathbb{L}^{2}\left([0,\tau]\times \Omega_n\right)$, using the fact that
	\begin{eqnarray*}
		\Vert u_m\Vert^2_{\mathbb{L}^{2}(\Omega_\tau)} &=& \displaystyle\int_{0}^{\tau}\displaystyle\int_{\Omega}  u_m(t,x)^{2} \,dx\, dt\\[5pt]
		& \leq & \text{mes}(\Omega)\tau,
	\end{eqnarray*}
	and 
	\begin{eqnarray*}
		\Vert v_m\Vert^2_{\mathbb{L}^{2}\left([0,\tau]\times \Omega_n\right)}& =& \displaystyle\int_{0}^{\tau}\displaystyle\int_{\Omega} \chi_{\Omega_n}(x) v_m(t,x)^2\,dx\,dt\\[5pt]
		&\leq&  \dfrac{ \text{mes}(\Omega_n)}{n^2}\tau,
	\end{eqnarray*}
	we can affirm that there exist subsequences of $(u_m)_{m\geq 0}$ and $(v_m)_{m\geq 0}$, still denoted by the same index $m\geq 0$ such that 
	\begin{center}
		$u_{m}\rightharpoonup u^*$ \text{ weakly in } $\mathbb{L}^{2}(\Omega_\tau)$ \text{ as } $m\to +\infty$,
	\end{center}
	and 
	\begin{center}
		$v_{m}\rightharpoonup v^*$ \text{ weakly in } $\mathbb{L}^{2}\left([0,\tau]\times \Omega_n\right)$ \text{ as } $m\to +\infty$.
	\end{center}
	Since $\mathcal{U}_{ad}$ is closed and convex, we get that $(u^*,v^*)\in \mathcal{U}_{ad}$.
	Writing 
	\begin{center}
		$y_{i,m}y_{j,m}-y^{*}_iy^{*}_j=(y_{i,m}-y_i^{*})y_{j,m}+y^{*}_i(y_{j,m}-y_j^{*})$,
	\end{center}
	we show that 
	\(F(t,y_m)\to F(t,y^*) \text{ in } \left(\mathbb{L}^{2}(\Omega_\tau)\right)^{6}.\)
	By taking $m \to +\infty$ in \eqref{equation 4.2}, we obtain that $y^{*}$ is the unique solution of the following equation:
	
	\begin{equation}
		\left\{\begin{array}{l}
			\dfrac{\partial y^{*}(t)}{\partial t}=By^{*}(t)+F(t,y^{*}(t)), \quad t\in[0,\tau]\\[8pt]
			y^{*}(0)=y_0,
		\end{array}\right.
	\end{equation}
	corresponding to $(u^{*},v^{*})\in \mathcal{U}_{ad}$. Since $\mathcal{J}(\cdot,\cdot,\cdot)$ is strictly convex (as the sum of a linear part and a quadratic part) and bounded (due to the boundedness of the solutions and controls), it follows by Proposition II.4.5 from \cite{44} that $\mathcal{J}(\cdot,\cdot,\cdot)$ is a proper, weakly lower semi-continuous functional. Therefore,
	\begin{eqnarray*}
		\mathcal{I}\leq \mathcal{J}(y^{*},u^{*},v^{*}) &\leq&  \liminf\limits_{n\to +\infty} \mathcal{J}(y_m,u_m,v_m)\\
		&\leq& \lim\limits_{n\to +\infty}\mathcal{J}(y_m,u_m,v_m)\\
		&=& \mathcal{I}.
	\end{eqnarray*}
	As a consequence $\mathcal{I}=\mathcal{J}(y^{*},u^{*},v^{*})$. The uniqueness of $(y^{*}, u^{*}, v^{*})$ follows from the strict convexity of $\mathcal{J}$ (see \cite[Remark 2.12 and Theorem 2.11]{barbu2012convexity} for more details).
\end{proof}

\section{Necessary Optimality Conditions} \label{section 5}
\noindent

Let $u^*$ and $v^*$ denote the optimal controls of system \eqref{eq 1}-\eqref{equ 2.3}, and set $w^* = (u^*, v^*)^T$. Let $w = (u, v)^T \in \mathcal{U}_{ad}$. In this section, we present the optimality conditions for system \eqref{eq 1}-\eqref{equ 2.3} and provide the characterization of the optimal controls. Let $y^{\epsilon}=(y_1^{\epsilon},\ldots,y_6^{\epsilon})$ and $y^{*}=(y_1^{*},\ldots,y_6^{*})$ be the solutions of \eqref{equation 2} corresponding to $w^{\epsilon}=w^{*}+\epsilon w\in \mathcal{U}_{ad}$ (for $\epsilon>0$ small enough) and $w^{*}$ respectively. Let $M^{*}(t)=\delta y^{*}_3(t)+(1-q)y^{*}_5(t)+\mu y^{*}_4(t)$.  We put
\begin{eqnarray*}
	H(t)&=&\begin{pmatrix}
		\dfrac{\partial F_i(t,y^{*}(t))}{\partial y_j}
	\end{pmatrix}_{1\leq i,j\leq 6}\\ \\
&=&
	\begin{pmatrix}
		-M^{*}(t)-\chi_{\Omega_n}v^{*}(t) & 0 & -\delta y^{*}_1(t) & -\mu y^{*}_1(t) & -(1-q)y^{*}_1(t) & \xi\\
		\chi_{\Omega_n}v^{*}(t) & 0 & 0 & 0 & 0 & 0\\
		M^{*}(t) & 0 & -k+\delta y^{*}_1(t) & \mu y^{*}_1(t) & (1-q)y^{*}_1(t) & 0\\
		0 & 0 & (1-z)k & -\eta & 0 & 0\\
		0 & 0 & zk & (1-p)\eta & -f-u^{*}(t) & 0\\
		0 & 0 & 0 & p\eta & \alpha f+u^{*}(t) & -\xi
	\end{pmatrix},
\end{eqnarray*}
and 
\begin{center}
	$G(t)=\begin{pmatrix}
		\dfrac{\partial F_1(t,y^*(t))}{\partial u} & \dfrac{\partial F_1(t,y^*(t))}{\partial v}\\[8pt]
		\vdots  & \vdots \\[8pt]
		\dfrac{\partial F_6(t,y^*(t))}{\partial u} & \dfrac{\partial F_6(t,y^*(t))}{\partial v}
	\end{pmatrix}
	=\begin{pmatrix}
		0 & -\chi_{\Omega_n}y_1^{*}(t)\\
		0 & \chi_{\Omega_n}y_1^{*}(t)\\
		0 & 0\\
		0 & 0\\
		-y^{*}_5(t) & 0\\
		y^{*}_5(t) & 0
	\end{pmatrix}$.
\end{center}
\begin{theorem}
	The mapping
	\[y:\mathcal{U}_{ad}\to \left(W^{1,2}\left(0,\tau;\mathbb{L}^{2}(\Omega)\right)\cap \mathbb{L}^{2}(0,\tau;H_0^{2}(\Omega))\cap \mathbb{L}^{\infty}(0,\tau;H^{1}(\Omega)) \right)^{6},\]
	is Gateaux differentiable with respect to $w^{*}$. Moreover, for all direction $w\in \mathcal{U}_{ad}$, $Y=y'(w^{*})(w)$ is the unique solution in
	
	\begin{center}
		$\left(W^{1,2}\left(0,\tau;\mathbb{L}^{2}(\Omega)\right)\cap \mathbb{L}^{2}(0,\tau;H_0^{2}(\Omega))\cap \mathbb{L}^{\infty}(0,\tau;H^{1}(\Omega)) \right)^{6}$
	\end{center}
	of the following equation:
	\begin{equation}
		\left\{\begin{array}{l}
			\dfrac{\partial Y(t)}{\partial t}=BY(t)+H(t)Y(t)+G(t)w(t), \quad t\in[0,\tau]\\[10pt]
			Y(0)=0.
		\end{array}\right.
		\label{equ 5.1}
	\end{equation}
\end{theorem}
\begin{proof}
	Let $Y_i^{\epsilon}=\dfrac{y_i^\epsilon-y_i^{*}}{\epsilon}$ for $i=1,2,\ldots,6$ and
	$M_\epsilon=(\delta y^\epsilon_3+(1-q)y^\epsilon_5+\mu y^\epsilon_4)$.
	We denote by $(S^{\epsilon})$ equation \eqref{equation 2} corresponding to $w^{\epsilon}=(
	u^{\epsilon},v^{\epsilon})^{T}$ and $(S^{*})$ equation \eqref{equation 2} corresponding to $w^{*}=(u^{*},v^{*})^{T}$. Subtracting $(S^{*})$ from $(S^{\epsilon})$, we obtain
	\begin{equation}
		\left\{\begin{array}{l}
			\dfrac{\partial Y^{\epsilon}(t)}{\partial t}=BY^{\epsilon}(t)+H^{\epsilon}(t)Y^{\epsilon}(t)+G^{\epsilon}(t)w(t), \quad t\in[0,\tau]\\[10pt]
			Y^{\epsilon}(0)=0,
		\end{array}\right.
		\label{equation 5.1}
	\end{equation}
	where 
	\begin{eqnarray*}
			H^{\epsilon}(t) &= &\begin{pmatrix}
		\dfrac{\partial F_i(t,y^{\epsilon}(t))}{\partial y_j}
		\end{pmatrix}_{1\leq i,j\leq 6}
	\end{eqnarray*}

	\begin{eqnarray*}
		&=& \begin{pmatrix}
			-M_\epsilon(t)-\chi_{\Omega_n}v^{\epsilon}(t) & 0 & -\delta y^{\epsilon}_1(t) & -\mu y^{\epsilon}_1(t) & -(1-q)y^{\epsilon}_1(t) & \xi\\
			\chi_{\Omega_n}v^{\epsilon}(t) & 0 & 0 & 0 & 0 & 0\\
			M_\epsilon(t) & 0 & -k+\delta y^{\epsilon}_1(t) & \mu y^{\epsilon}_1(t) & (1-q)y^{\epsilon}_1(t) & 0\\
			0 & 0 & (1-z)k & -\eta & 0 & 0\\
			0 & 0 & zk & (1-p)\eta & -f-u^{\epsilon}(t) & 0\\
			0 & 0 & 0 & p\eta & \alpha f+u^{\epsilon}(t) & -\xi
		\end{pmatrix},
	\end{eqnarray*}
	and 
	\begin{center}
		$G^{\epsilon}(t)=\begin{pmatrix}
		\dfrac{\partial F_1(t,y^\epsilon(t))}{\partial u} & \dfrac{\partial F_1(t,y^\epsilon(t))}{\partial v}\\[8pt]
		\vdots & \vdots \\[8pt]
		\dfrac{\partial F_6(t,y^\epsilon(t))}{\partial u} & \dfrac{\partial F_6(t,y^\epsilon(t))}{\partial v}
		\end{pmatrix}=\begin{pmatrix}
			0 & -\chi_{\Omega_n}y_1^{\epsilon}(t)\\
			0 & \chi_{\Omega_n}y_1^{\epsilon}(t)\\
			0 & 0\\
			0 & 0\\
			-y^{\epsilon}_5(t) & 0\\
			y^{\epsilon}_5(t) & 0
		\end{pmatrix}$.
	\end{center}
	The solution of equation \eqref{equation 5.1} can be written as follows:
	\begin{equation}
		Y^{\epsilon}(t)=\displaystyle\int_{0}^{t}T(t-s)H^{\epsilon}(s)Y^{\epsilon}(s)ds+\displaystyle\int_{0}^{t}T(t-s)G(s)w(s)ds, \quad t\in[0,\tau].
		\label{equ 5.3}
	\end{equation}
	Since $\Vert y^\epsilon_i\Vert_{\mathbb{L}^{\infty}(\Omega_\tau)}\leq C $ ($C$ is independent of $\epsilon>0$), it follows that the coefficients of matrix $H^{\epsilon}(t)$ and $G^{\epsilon}(t)$ are uniformly bounded with respect to $\epsilon>0$. Using Gronwall's Lemma, we can affirm that there exists $\varGamma>0$ ($\varGamma$ is independent of $\epsilon>0$) such that 
	\begin{center}
		$\Vert Y^{\epsilon}_i\Vert_{\mathbb{L}^{2}(\Omega_\tau)}\leq \varGamma$ \text{ for } $i=1,2,\ldots,6$.
	\end{center}
	Then, 
	\begin{center}
		$\Vert y_i^{\epsilon}-y_i^{*}\Vert_{\mathbb{L}^{2}(\Omega_\tau)}=\epsilon \Vert Y^{\epsilon}_i\Vert_{\mathbb{L}^{2}(\Omega_\tau)}\leq \epsilon \varGamma$ \text{ for } $i=1,2,\ldots,6$.
	\end{center}
	Hence, $y_i^{\epsilon}\to y_i^{*}$ in $\mathbb{L}^{2}(\Omega_\tau)$  as $\epsilon\to 0$ for $i=1,2,\ldots,6$. Let $Y(t)=(Y_1(t), \ldots,\break Y_6(t))$ be the solution of equation \eqref{equ 5.1}. Then,
	\begin{equation}
		Y(t)=\displaystyle\int_{0}^{t}T(t-s)H(s)Y(s)ds+\displaystyle\int_{0}^{t}T(t-s)G(s)Y(s)ds, \quad t\in[0,\tau].
		\label{equ 5.4}
	\end{equation}
	By \eqref{equ 5.3} and \eqref{equ 5.4}, we have
	\begin{eqnarray*}
		Y^{\epsilon}(t)-Y(t)&=&\displaystyle\int_{0}^{t}T(t-s)(H^{\epsilon}(s)-H(s))Y(s)ds\\
		& & +\displaystyle\int_{0}^{t}T(t-s)H^{\epsilon}(s)(Y^{\epsilon}(s)-Y(s))ds\\
		& & + \hspace{0.1cm}\displaystyle\int_{0}^{t}T(t-s)(G^{\epsilon}(s)-G(s))w(s)ds,
	\end{eqnarray*}
	for $t\in[0,\tau]$. Then,
	
	\begin{eqnarray*}
		\Vert Y^{\epsilon}(t)-Y(t)\Vert&\leq& \displaystyle\int_{0}^{t}\Vert H^{\epsilon}(s)-H(s)\Vert_{\mathcal{L}(H)}\Vert Y(s)\Vert ds\hspace{0.1cm} +
	\end{eqnarray*}

\begin{eqnarray*}
		& &  \displaystyle\int_{0}^{t} \Vert G^{\epsilon}(s)-G(s)\Vert_{\mathcal{L}(U,H)}\Vert w(s)\Vert_Uds\\
		& & + \hspace{0.1cm} \displaystyle\int_{0}^{t} \Vert H^{\epsilon}(s)\Vert_{\mathcal{L}(H)}\Vert Y^{\epsilon}(s)-Y(s)\Vert ds,
	\end{eqnarray*}
	for $t\in[0,\tau]$. Because of the boundedness of $y^\epsilon_i$ for $i=1,\ldots,6$, uniformly with respect to $\epsilon > 0$, we show that $H^{\epsilon}(s)$ is uniformly bounded with respect to $\epsilon > 0$. Let $\tilde{C} > 0$ such that $H^{\epsilon}(s) \leq \tilde{C}$ for all $s \in [0, \tau]$. Since 
	\[
	\Vert H^{\epsilon}(s) - H(s) \Vert_{\mathcal{L}(H)} \to 0 \quad \text{and} \quad \Vert G^{\epsilon}(s) - G(s) \Vert_{\mathcal{L}(U, H)} \to 0
	\]
	for almost every $s \in [0, t]$ as $\epsilon \to 0$, for sufficiently small $\epsilon > 0$, we apply the dominated convergence theorem to obtain
	\[
	\Vert Y^{\epsilon}(t) - Y(t) \Vert \leq \dfrac{\epsilon}{2} + \dfrac{\epsilon}{2} + \int_{0}^{t} \tilde{C} \Vert Y^{\epsilon}(s) - Y(s) \Vert \, ds.
	\]
	Using Gronwall's Lemma, we deduce
	
	\begin{center}
		$\Vert Y^{\epsilon}(t) - Y(t) \Vert \leq \epsilon e^{\tilde{C} t}$.
	\end{center}
	Consequently, the convergence $Y^{\epsilon} \to Y$ in $\mathbb{L}^{2}(\Omega_\tau)$ as $\epsilon \to 0$ follows from the fact that
	\[
	\Vert Y^{\epsilon} - Y \Vert_{\mathbb{L}^{2}(\Omega_\tau)} \leq \dfrac{\epsilon}{\sqrt{2 \tilde{C}}} e^{\tilde{C} \tau}.
	\]
\end{proof}

Let 
\(\rho(x)=\left(\chi_{\Omega_n}(x)\rho_1,
0,
\rho_3,
\rho_4,
\rho_5,
0\right)^{T}\text{ for } x\in \Omega.\) The following theorem provides the characterization of the optimal controls.
\begin{theorem} Let $(u^{*},v^{*})$ denote the optimal controls for the system \eqref{eq 1}-\eqref{equ 2.3}, and let $y^{*}$ be the optimal state of equation \eqref{equation 2} corresponding to $(u^{*},v^{*})$. Then,
	\[
	u^{*}(t,x) = \max\left(0, \min\left(1, \dfrac{y^{*}_5(t,x)(p_5(t,x) - p_6(t,x))}{\sigma_1}\right)\right),
	\]
	and
	\[
	v^{*}(t,x) = \max\left(0, \min\left(\dfrac{1}{n}, \frac{y_1^{*}(t,x)(p_1(t,x) - p_2(t,x))}{\sigma_2}\right)\right),
	\]
	where $P = \left(p_1, p_2, p_3, p_4, p_5, p_6\right)$ is the solution of the following adjoint equation:
	\begin{equation}\label{equ 5.5}
	\left\{
	\begin{array}{l}
	-\dfrac{\partial P(t,x)}{\partial t} - BP(t,x) - H^{*}(t,x)P(t,x) = \rho(x), \quad t \in [0,\tau], \\[8pt]
	P(\tau,x) = 0, \quad x \in \Omega,
	\end{array}
	\right.
	\end{equation}
	where $H^{*}(t,x) = H^{*}(t)(x)$ for $(t,x) \in \Omega_\tau$.
\label{theorem 5.2}
\end{theorem}
\begin{proof}
	Let $\varepsilon>0$, $w^{*}=(u^{*},v^{*})^{T}$ and $w^{\varepsilon}=(u^{\epsilon},v^{\epsilon})^{T}=w^{*}+\varepsilon h$ for $h=(h_1,h_2)^{T}$. 
	Then,
	\begin{eqnarray*}
		\mathcal{J}'(y^{*},u^{*},v^{*})(h)
		&=& \lim\limits_{\varepsilon\to 0}\dfrac{\mathcal{J}(y^{\varepsilon},u^{\varepsilon},v^{\varepsilon})-\mathcal{J}(y^{*},u^{*},v^{*})}{\varepsilon}\\
		&=& \lim\limits_{\varepsilon\to 0}\left[\displaystyle\int_{0}^{\tau}\displaystyle\int_{\Omega}\rho_1\chi_{\Omega_n}(x)\left(\dfrac{y^{\varepsilon}_1(t,x)-y^{*}_1(t,x)}{\varepsilon}\right)\,dx\,dt\right.\\
		& & + \hspace{0.1cm}\left.\displaystyle\int_{0}^{\tau}\displaystyle\int_{\Omega}\sum\limits_{i=3}^{5}\left(\rho_i\left(\dfrac{y^{\varepsilon}_i(t,x)-y^{*}_i(t,x)}{\varepsilon}\right)\right)\,dx\,dt\right.\\
		& & +\hspace{0.1cm}\left.\dfrac{\sigma_1}{2}\displaystyle\int_{0}^{\tau}\displaystyle\int_{\Omega}\left(\dfrac{(u^{\varepsilon}(t,x))^{2}-(u^{*}(t,x))^{2}}{\varepsilon}\right)\,dx\,dt\right.\\
		& & + \hspace{0.1cm}\left.\dfrac{\sigma_2}{2}\displaystyle\int_{0}^{\tau}\displaystyle\int_{\Omega}\chi_{\Omega_n}(x)\left(\dfrac{(v^{\varepsilon}(t,x))^{2}-(v^{*}(t,x))^{2}}{\varepsilon}\right)\,dx\,dt\right]\\
		&=& \lim\limits_{\varepsilon\to 0}\left[\displaystyle\int_{0}^{\tau}\displaystyle\int_{\Omega}\rho_1\chi_{\Omega_n}(x)\left(\dfrac{y^{\varepsilon}_1(t,x)-y^{*}_1(t,x)}{\varepsilon}\right)\,dx\,dt\right.\\
		& & + \hspace{0.1cm}\left.\displaystyle\int_{0}^{\tau}\displaystyle\int_{\Omega}\sum\limits_{i=3}^{5}\left(\rho_i\left(\dfrac{y^{\varepsilon}_i(t,x)-y^{*}_i(t,x)}{\varepsilon}\right)\right)\,dx\,dt\right.
	\end{eqnarray*}
\begin{eqnarray*}
		& & +\hspace{0.1cm}\left.\dfrac{\sigma_1}{2}\displaystyle\int_{0}^{\tau}\displaystyle\int_{\Omega}\left(2u^{*}(t,x)h_1+\varepsilon h_1^2\right)\,dx\,dt\right.\\
		& & + \hspace{0.1cm}\left.\dfrac{\sigma_2}{2}\displaystyle\int_{0}^{\tau}\displaystyle\int_{\Omega}\chi_{\Omega_n}(x)\left(2v^{*}(t,x)h_2+\varepsilon h_2^2\right)\,dx\,dt\right]\\
			&=& \displaystyle\int_{0}^{\tau}\displaystyle\int_{\Omega}\rho_1\chi_{\Omega_n}(x)Y_1(t,x)\,dx\,dt+\displaystyle\int_{0}^{\tau}\displaystyle\int_{\Omega} \sum\limits_{i=3}^{5}\rho_iY_i(t,x)\,dx\,dt\\
		& & + \hspace{0.1cm}\displaystyle\int_{0}^{\tau}\displaystyle\int_{\Omega}\left(\sigma_1u^{*}(t,x)h_1(t,x)+\sigma_2\chi_{\Omega_n}(x)v^{*}(t,x)h_2(t,x)\right)\,dx\,dt\\
		&=& \displaystyle\int_{0}^{\tau}\langle \rho,Y\rangle_{H}dt+\displaystyle\int_{0}^{\tau}\langle \bar{w}^{*},h\rangle_U \,dt,
	\end{eqnarray*}
	with $\bar{w}^{*}=\begin{pmatrix}
		\sigma_1 u^{*}\\
		\sigma_2 \chi_{\Omega_n} v^{*}
	\end{pmatrix}$. Let $w\in \mathcal{U}_{ad}$, we take $h=w-w^{*}$, then
	\begin{center}
		$\mathcal{J}'(y^{*},u^{*},v^{*})(w-w^{*})
		= \displaystyle\int_{0}^{\tau}\langle \rho,Y\rangle_{H}dt+\displaystyle\int_{0}^{\tau}\langle \bar{w}^{*},w-w^{*}\rangle_U\, dt.$
	\end{center}
	We have 
	\begin{eqnarray*}
		\displaystyle\int_{0}^{\tau}\langle\rho,Y\rangle_{H}\,dt &=& \displaystyle\int_{0}^{\tau}\langle -\dfrac{\partial P}{\partial t}-BP-H^{*}P,Y\rangle_{H}\,dt\\
		&=& \displaystyle\int_{0}^{\tau}\langle P,\dfrac{\partial Y}{\partial t}-BY-HY\rangle_{H}\,dt\\
		&=& \displaystyle\int_{0}^{\tau}\langle P,G(w-w^{*})\rangle_{H}\,dt\\
		&=& \displaystyle\int_{0}^{\tau}\langle G^{*}P,w-w^{*}\rangle_U\,dt.
	\end{eqnarray*}
	Since $\mathcal{J}$ is convex, it follows that 
	\begin{center}
		$\mathcal{J}'(y^{*},u^{*},v^{*})(w-w^{*})=\displaystyle\int_{0}^{\tau}\langle G^{*}P+\bar{w}^{*},w-w^{*}\rangle_U\,dt\geq 0$.
	\end{center}
	By standard argument varying $w$, we obtain 
	\begin{eqnarray*}
		\bar{w}^{*}&=&-G^{*}P\\
		&=&- \begin{pmatrix}
			0 & 0 & 0 & 0 & -y^{*}_5 & y^{*}_5\\
			-\chi_{\Omega_n}y_1^{*} &  \chi_{\Omega_n}y_1^{*} & 0 & 0 & 0 & 0
		\end{pmatrix}\begin{pmatrix}
		\vspace{-0.2cm}
			p_1 \\ 
			\vspace{-0.2cm}
			\vdots\\
			p_6
		\end{pmatrix}.
	\end{eqnarray*}
	Then, 
	\begin{center}
		$	\begin{pmatrix}
			u^{*}\\
			v^{*}
		\end{pmatrix}=\begin{pmatrix}
			\dfrac{y^{*}_5(p_5-p_6)}{\sigma_1}\\[8pt]
			\dfrac{y_1^{*}(p_1-p_2)}{\sigma_2 }
		\end{pmatrix}$.
	\end{center}
	As $(u^{*},v^{*})\in \mathcal{U}_{ad}$, we write
	\[u^{*}(t,x)=\max\left(0,\min\left(1,\dfrac{y^{*}_5(t,x)(p_5(t,x)-p_6(t,x))}{\sigma_1}\right)\right),\]
	and 
	\[v^{*}(t,x)=\max\left(0,\min\left(\dfrac{1}{n},\dfrac{y_1^{*}(t,x)(p_1(t,x)-p_2(t,x))}{\sigma_2}\right)\right),\]
	for $t\in[0,\tau]$ and $x\in \Omega$.
\end{proof}
\section{Numerical simulation and some discussions} \label{section 6}
\noindent

This section applies our theoretical framework to the COVID-19 pandemic and demonstrates its effectiveness through computational simulations. For this purpose, the system of equations given by \eqref{eq 1}-\eqref{equ 2.3} is numerically solved using the explicit finite difference method. This yields approximations for the state variables, denoted by \( y_i^* \) for \( i=1,\ldots,6 \). Subsequently, the adjoint system \eqref{equ 5.5} is discretized and solved using the same numerical scheme. The resulting adjoint variables, \( p_i \) for \( i=1,\ldots,6 \), are employed to iteratively update the control functions according to the conditions specified in Theorem \ref{theorem 5.2}. The iterative process terminates when the difference between the current and previous values for the control variables is within an acceptable error range. Our investigations use a set of parameters from prior studies \cite{Li,Riou}, with some values specifically chosen for COVID-19. For simplicity, we set \( n = 1 \), as the methodology remains consistent for \( n \geq 2 \). The parameter values are summarized in Table \ref{table1_Num}.
\begin{table}[h!]
	\centering
	\caption{Model parameters}\label{table1_Num}
		\begin{tabular}{ccc} 
			\hline
			Parameter & \hspace{1.5cm} & Value\\
			\hline
			$\eta$ &  & 0.3/day\\
			$q$ &  & 0.9995\\ 
			$\xi$ &  & 0.001\\
			$\delta$ & & $10^{-5}$\\ 
			$\mu$ &  & $10^{-5}$\\ 
			$z$ &  & 0.1\\ 
			\hline
	\end{tabular}
	\hspace{0.5cm}
	\begin{tabular}{ccc} 
		\hline
		Parameter & \hspace{1.5cm} & Value\\
		\hline 
		$k$ &  & 0.54/day\\
		$f$ &  & 0.3/day\\ 
		$\tau$ &  & 20 days\\
		$\alpha$ & & 0.995\\
		$D_i$ & & 0.001\\
		$p$ &  & 0.02\\
		\hline
	\end{tabular}
\end{table}

\normalsize
 For the initial state variables, we take
\begin{equation*}
\begin{aligned}
S_0(x) &= 4000 \sin (\pi x) + 8000 \left(1 - \frac{1}{\pi}\right), \quad & x \in (0,1),\\
Q_0(x) &= 0,\quad &  x \in (0,1),\\
E_0(x) &= 100 \exp(x) + 282.2, \quad &  x \in (0,1),\\
A_0(x) &= 500 \cos (\pi x) + 500, \quad & x \in (0,1),\\
I_0(x) &= 500 \cos (\pi x) + 500, \quad &  x \in (0,1),\\
R_0(x) &= 0, \quad &  x \in (0,1).
\end{aligned}
\end{equation*}

Figures \ref{fig2} to \ref{fig13} illustrate the variation in the number of individuals in each group, both with and without controls. Figure \ref{fig14} depicts the absolute difference between the total populations with and without controls. Additionally, Figures \ref{fig15} and \ref{fig16} present the evolution of treatment and quarantine controls, respectively.

Figure \ref{fig2} illustrates the evolution of susceptible individuals in the absence of control measures. It is evident that their numbers decline sharply, reaching zero within approximately one week. This rapid depletion of the susceptible population indicates that these individuals are transitioning into the exposed category, subsequently contributing to an increase in asymptomatic and infected cases. The dynamics change significantly when control measures are implemented. As shown in Figures \ref{fig3} and \ref{fig5}, a portion of the susceptible population is moved into the quarantined group. This intervention reduces the number of susceptible individuals transitioning directly to the exposed group. Notably, within five days, some of the susceptible individuals still join the exposed group, although the overall impact is mitigated compared to the scenario without controls.
\begin{figure}[h!]
	\centering
	\begin{overpic}[width=\textwidth]{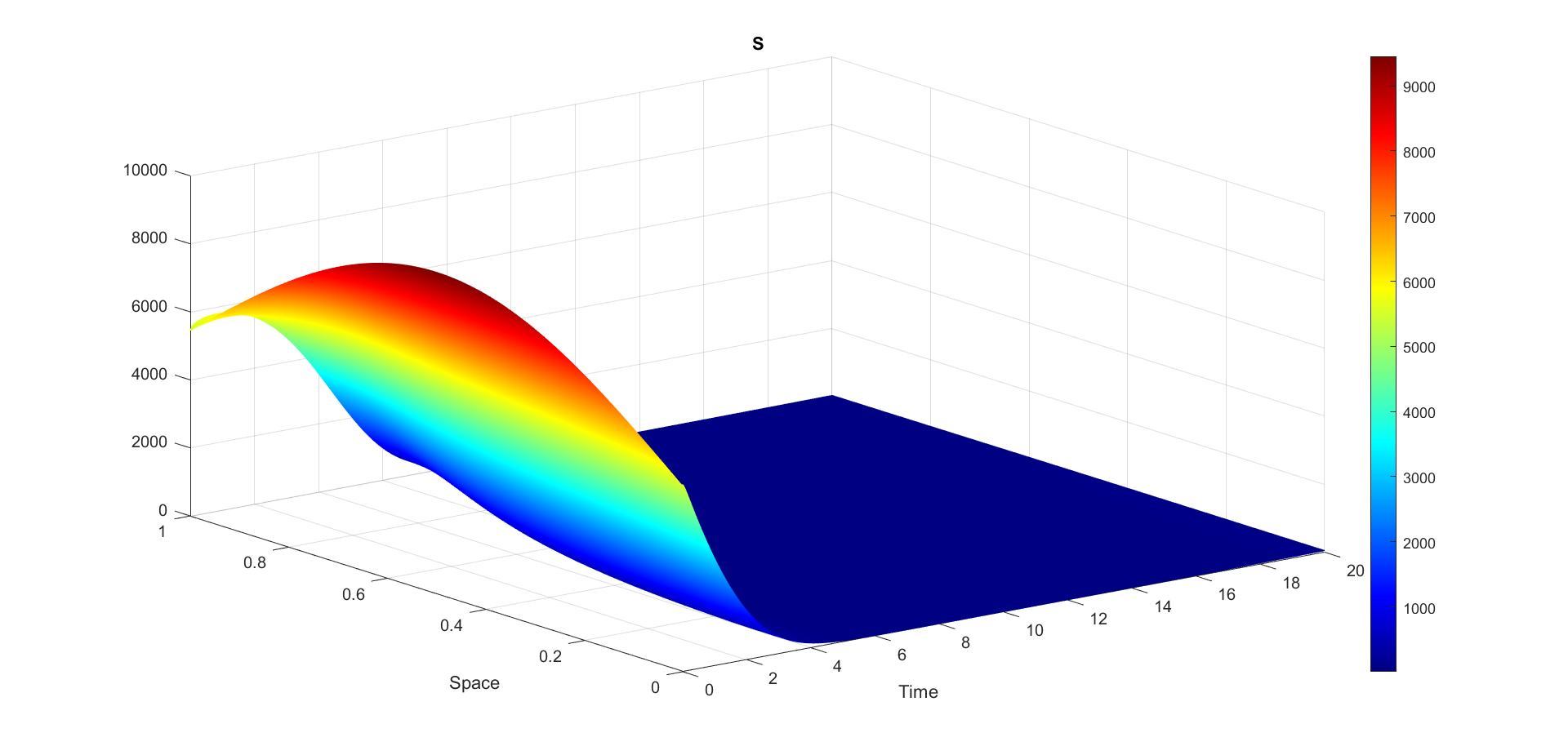}
		\put(62.3, 34.4){%
			\subfloat{%
				\begin{minipage}{0.5\textwidth}
					\includegraphics[width=0.5\textwidth]{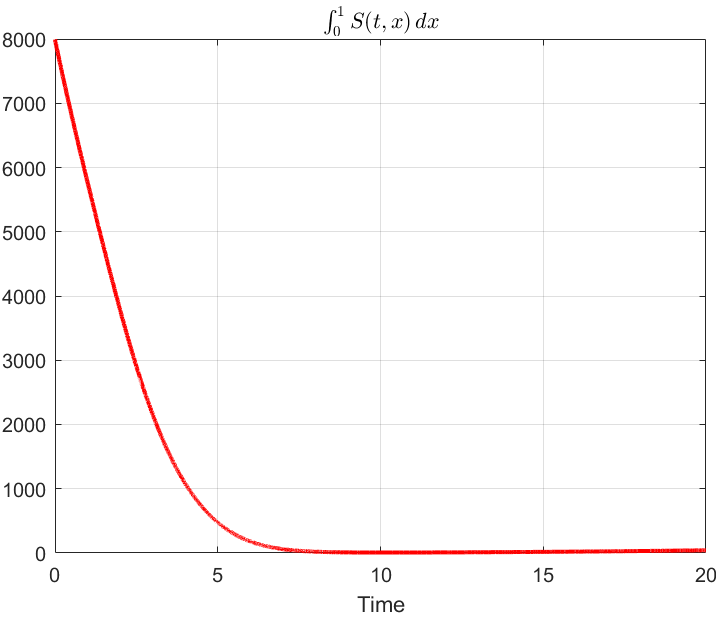}
				\end{minipage}
			}
		}
	\end{overpic}
	\caption{Changes in the Susceptible Group Without Control Actions.}
	\label{fig2}
\end{figure}
In the scenario without controls, Figure \ref{fig6} demonstrates that the number of exposed individuals begins to rise from the first day. Initially, there are only 500 exposed individuals, but this number escalates dramatically, surpassing 2500 cases within approximately ten days. Following this peak, the number of exposed individuals starts to decline as they progress to the infected and asymptomatic categories. Conversely, when control measures are applied, as depicted in Figure \ref{fig7}, the number of exposed individuals still increases, but the growth is significantly curtailed. The maximum number of exposed individuals does not exceed 1000 cases, which is a stark contrast to the uncontrolled scenario. This controlled increase is expected, given that direct control measures targeting the exposed group are not in place. The presence of controls indirectly influences the exposed population by reducing the influx from the susceptible group.
\begin{figure}[h!]
	\centering
	\begin{overpic}[width=\textwidth]{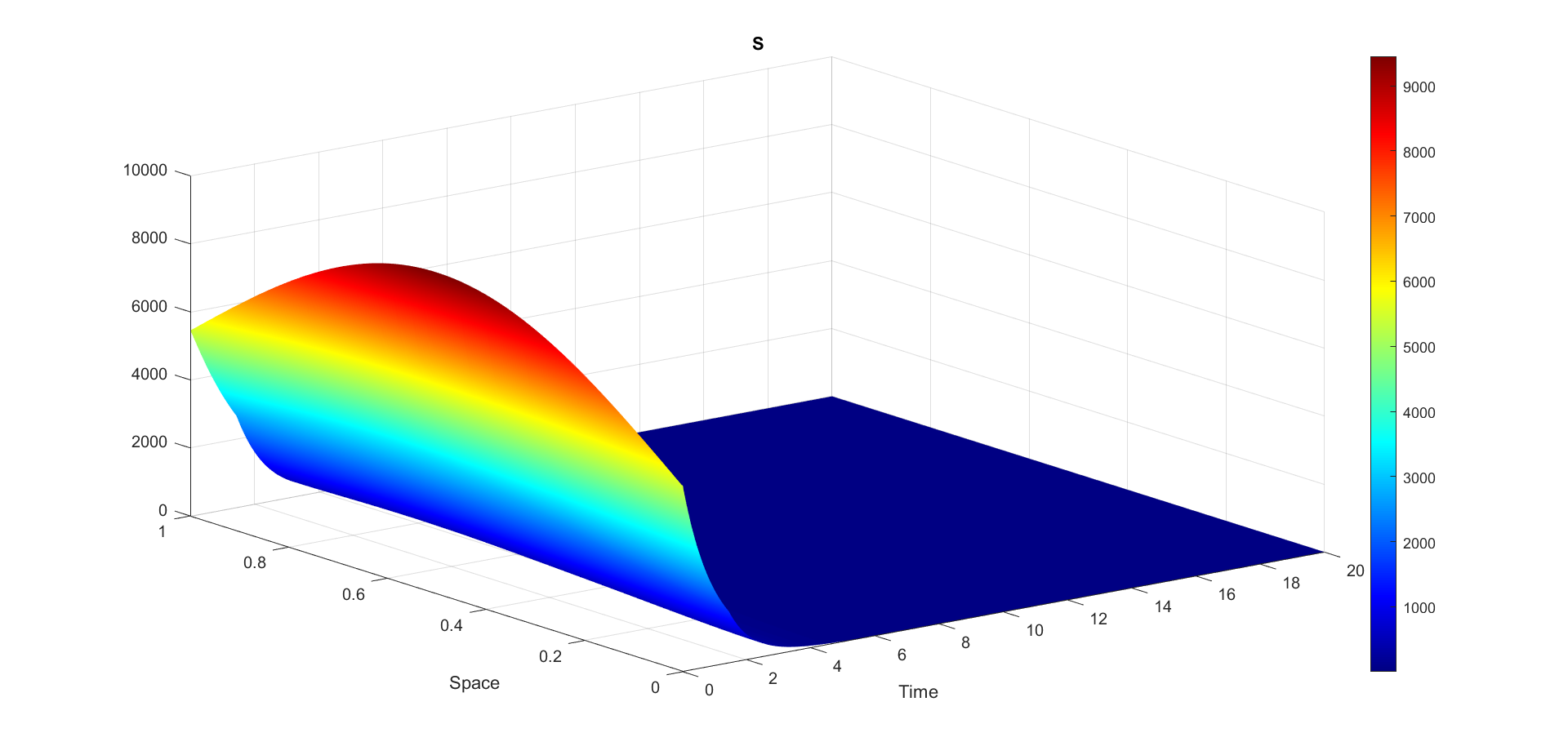}
		\put(62.3, 34.4){%
			\subfloat{%
				\begin{minipage}{0.5\textwidth}
					\includegraphics[width=0.5\textwidth]{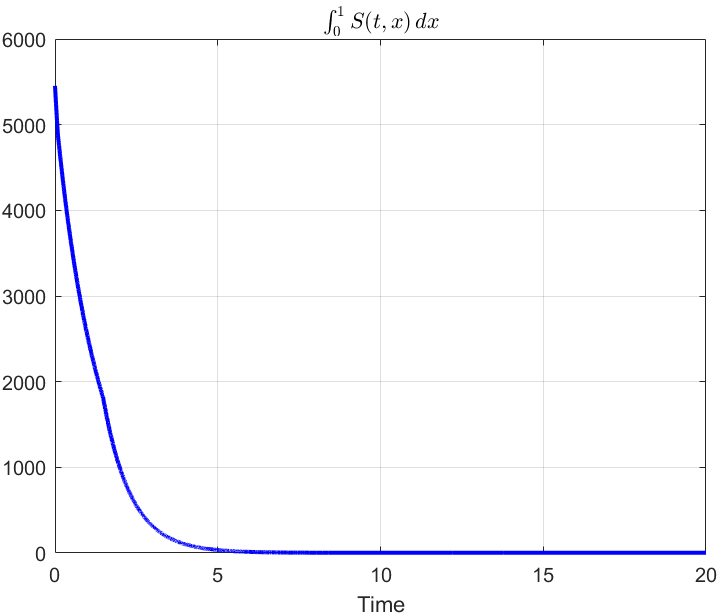}
				\end{minipage}
			}
		}
	\end{overpic}
	\caption{Changes in the Susceptible Group in the Presence of Controls.}
	\label{fig3}
\end{figure}

\begin{figure}[h!]
	\centering
	\includegraphics[scale=0.28]{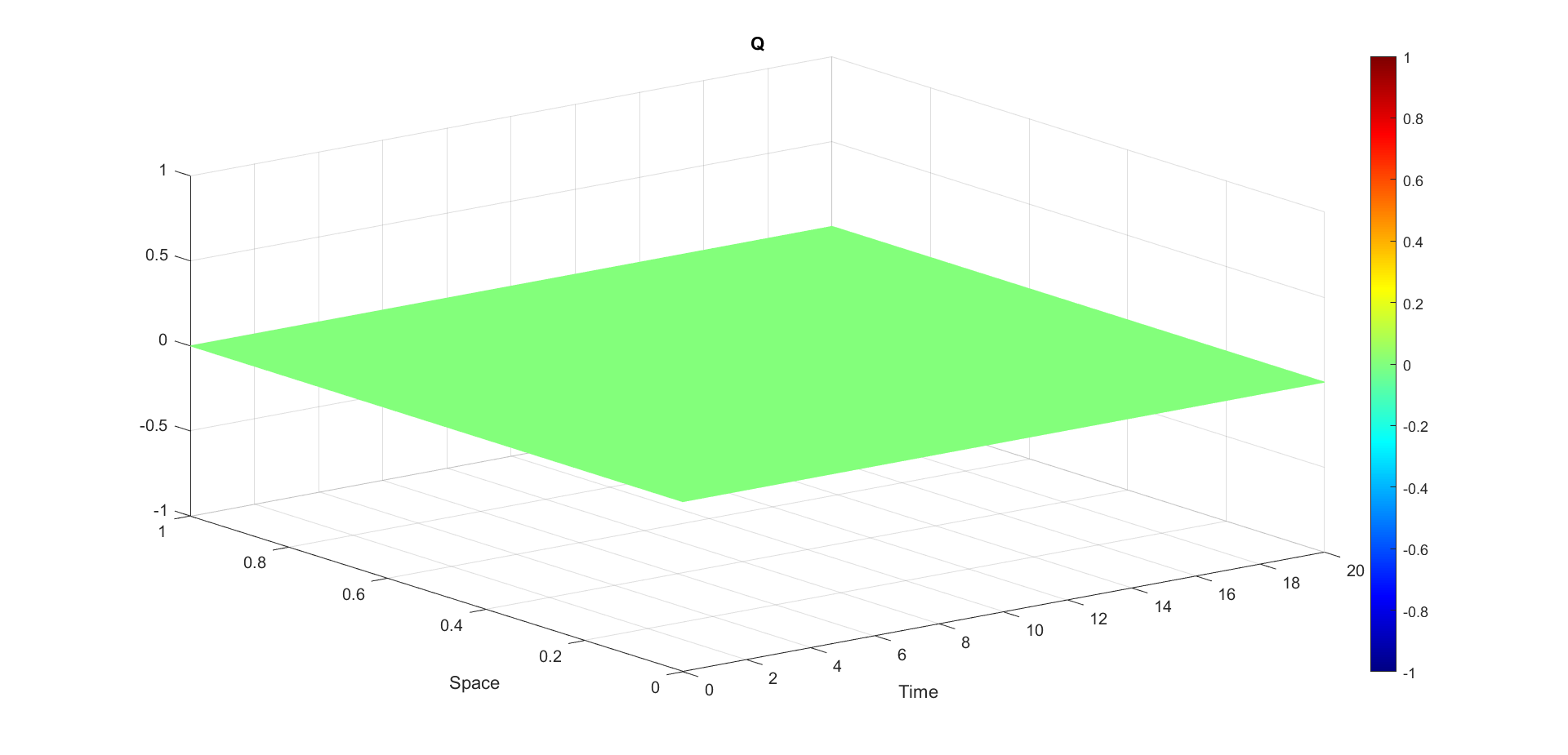}
	\caption{Changes in the Quarantined Group Without Control Actions.}
	\label{fig4}
\end{figure}
\begin{figure}[h!]
	\centering
	\includegraphics[scale=0.28]{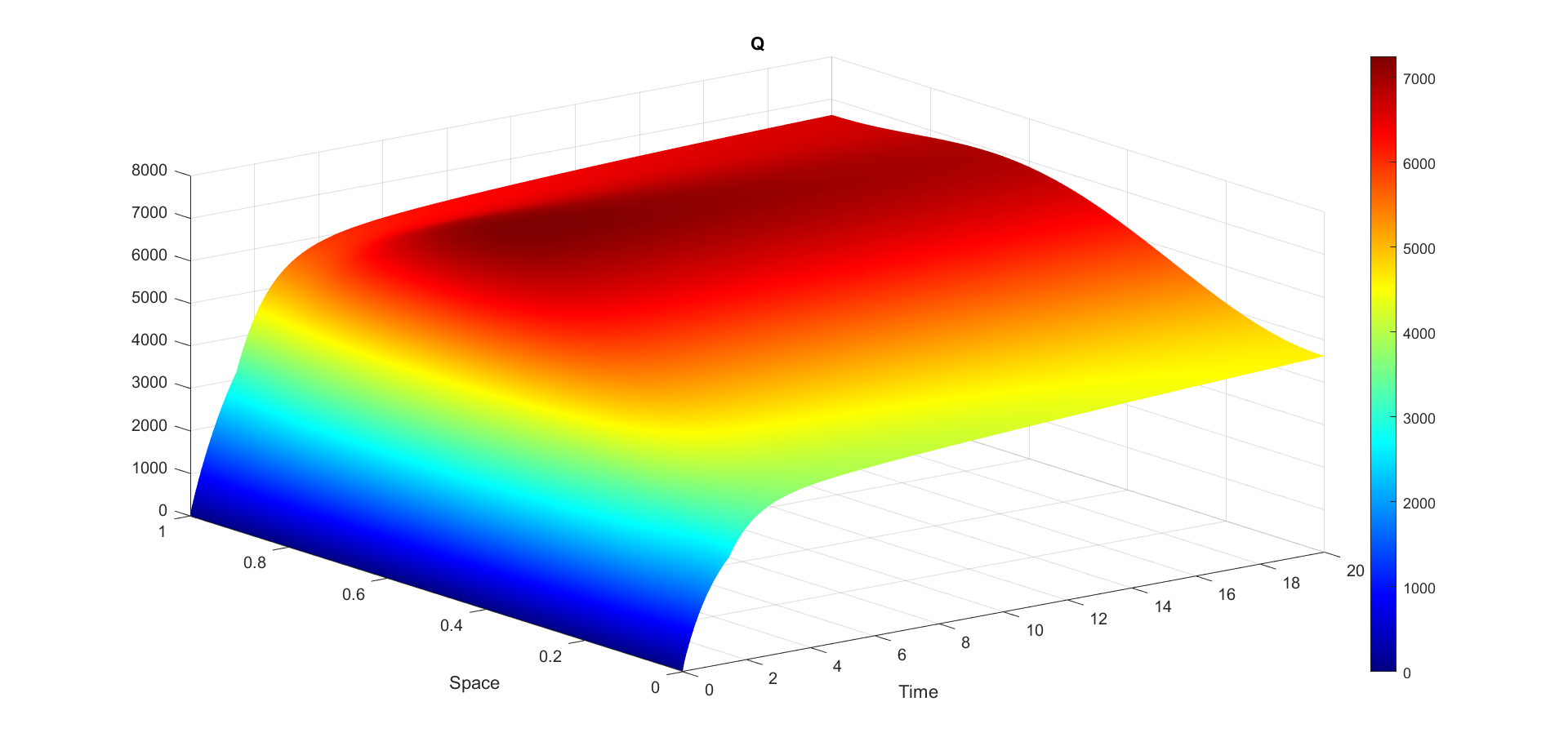}
	\caption{Changes in the Quarantined Group in the Presence of Controls.}
	\label{fig5}
\end{figure}

Figure \ref{fig8} shows a significant increase in the number of asymptomatic individuals in the absence of control measures. Specifically, the number of asymptomatic cases rises sharply, reaching approximately 3000 persons within the first and second weeks. This substantial increase can be attributed to the lack of intervention, allowing the virus to spread unchecked among the population. The growth in asymptomatic individuals is particularly concerning because these individuals, while not exhibiting symptoms, are still capable of transmitting the virus to others. This hidden spread contributes to a higher number of infected individuals due to the transmission rate $(1-p)\eta$.
\begin{figure}[h]
	\centering
	\begin{overpic}[width=\textwidth]{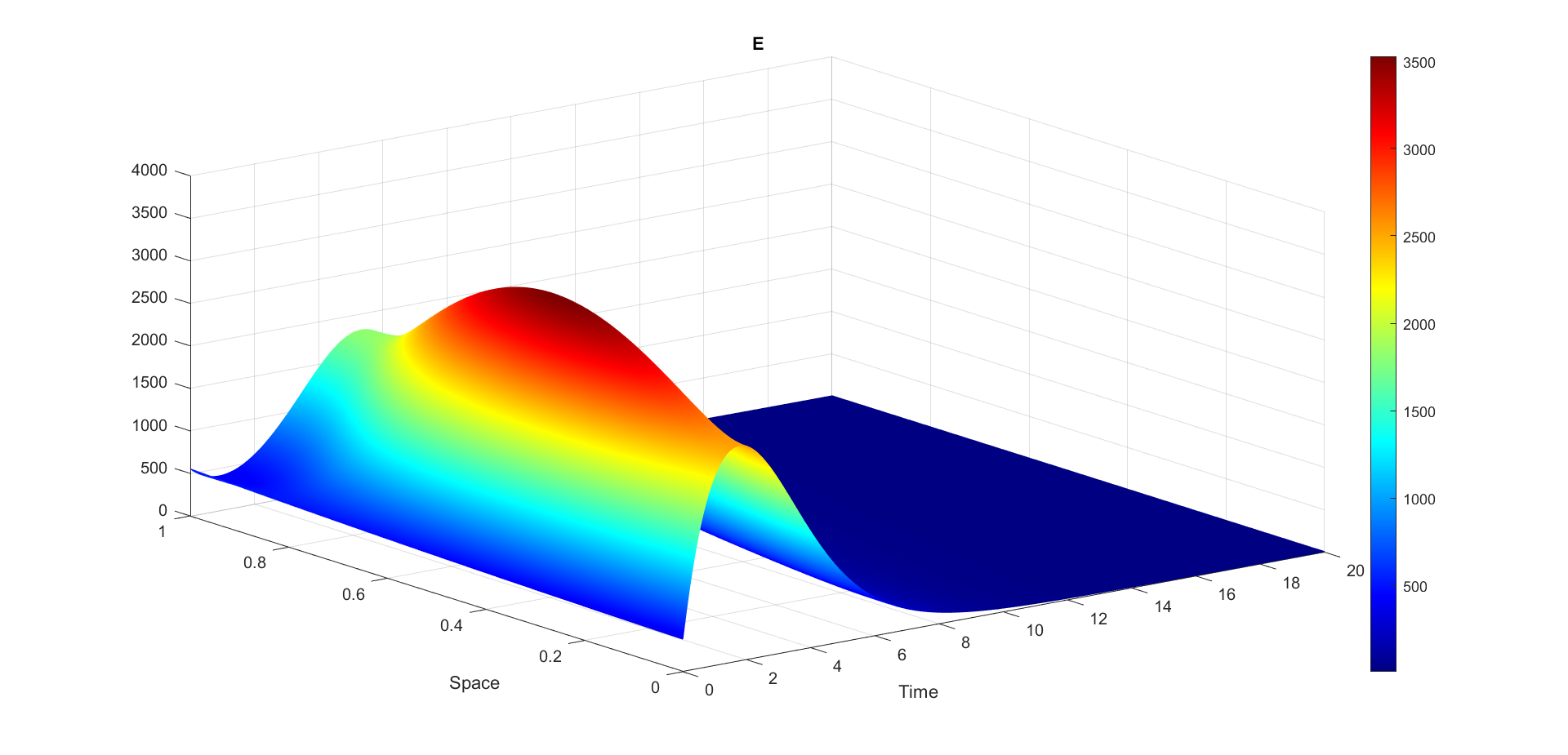}
		\put(62.3, 34.4){%
			\subfloat{%
				\begin{minipage}{0.5\textwidth}
					\includegraphics[width=0.5\textwidth]{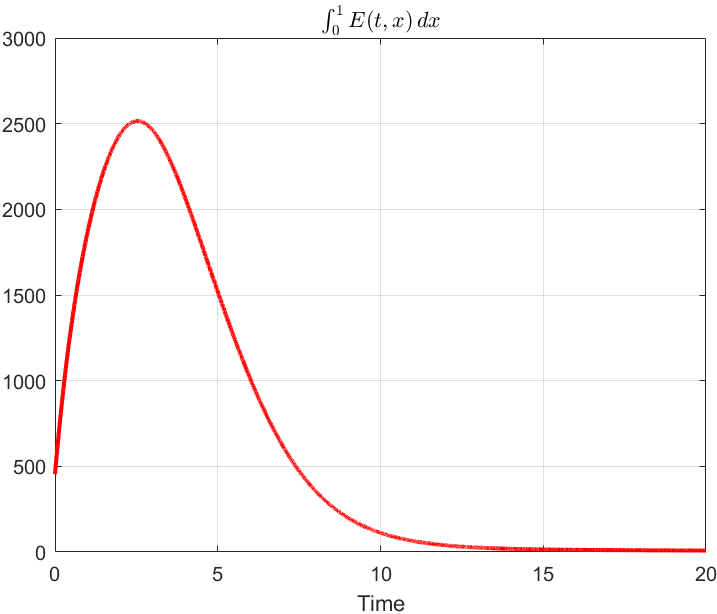}
				\end{minipage}
			}
		}
	\end{overpic}
	\caption{Changes in the Exposed Group Without Control Actions.}
	\label{fig6}
\end{figure}
\begin{figure}[h!]
	\centering
	\begin{overpic}[width=\textwidth]{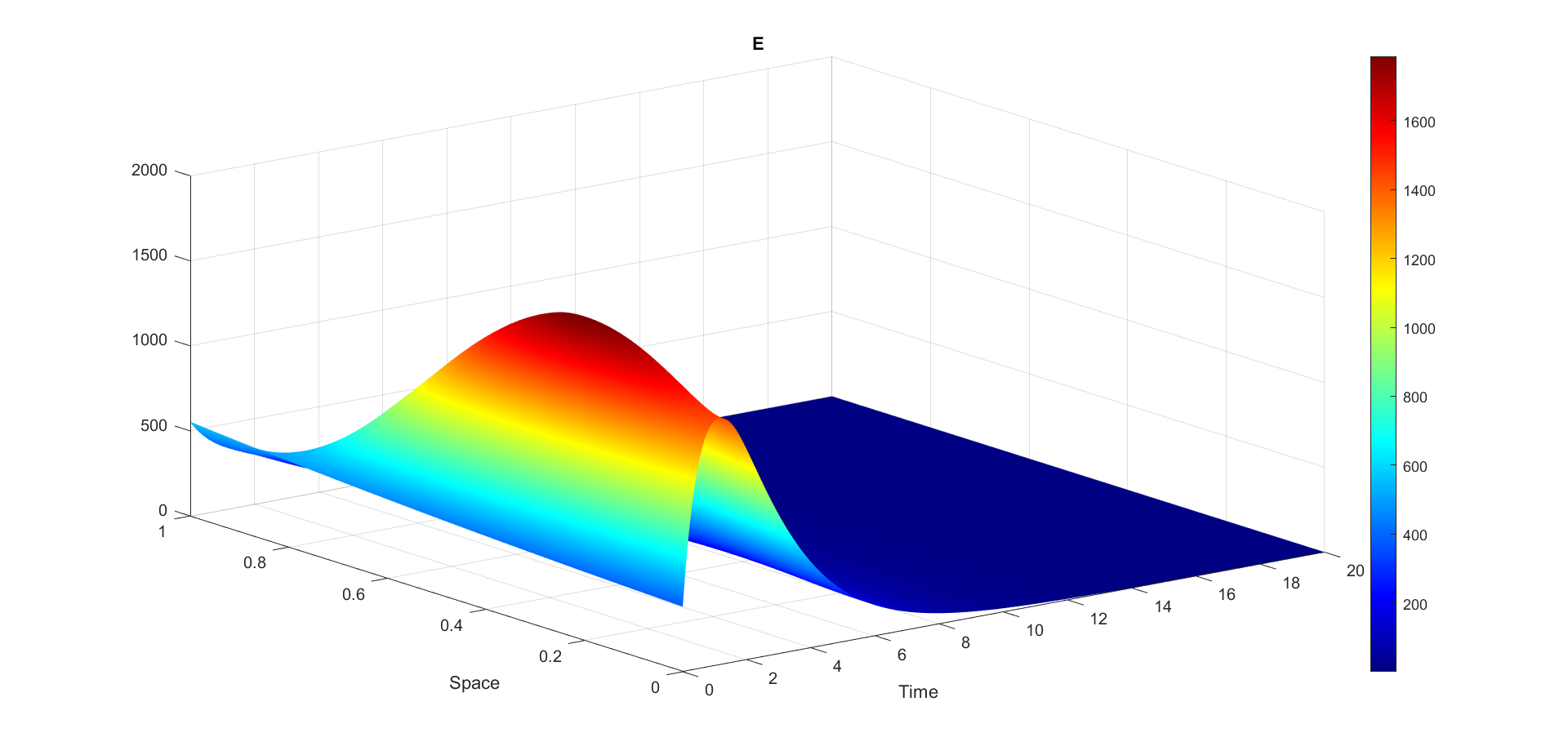}
		\put(62.3, 34.4){%
			\subfloat{%
				\begin{minipage}{0.5\textwidth}
					\includegraphics[width=0.5\textwidth]{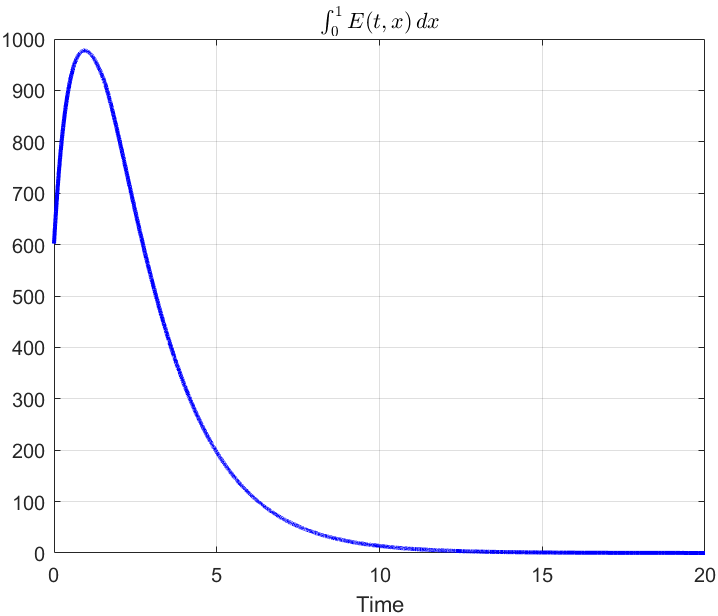}
				\end{minipage}
			}
		}
	\end{overpic}
	\caption{Changes in the Exposed Group in the Presence of Controls.}
	\label{fig7}
\end{figure}
\begin{figure}[h!]
	\centering
	\begin{overpic}[width=\textwidth]{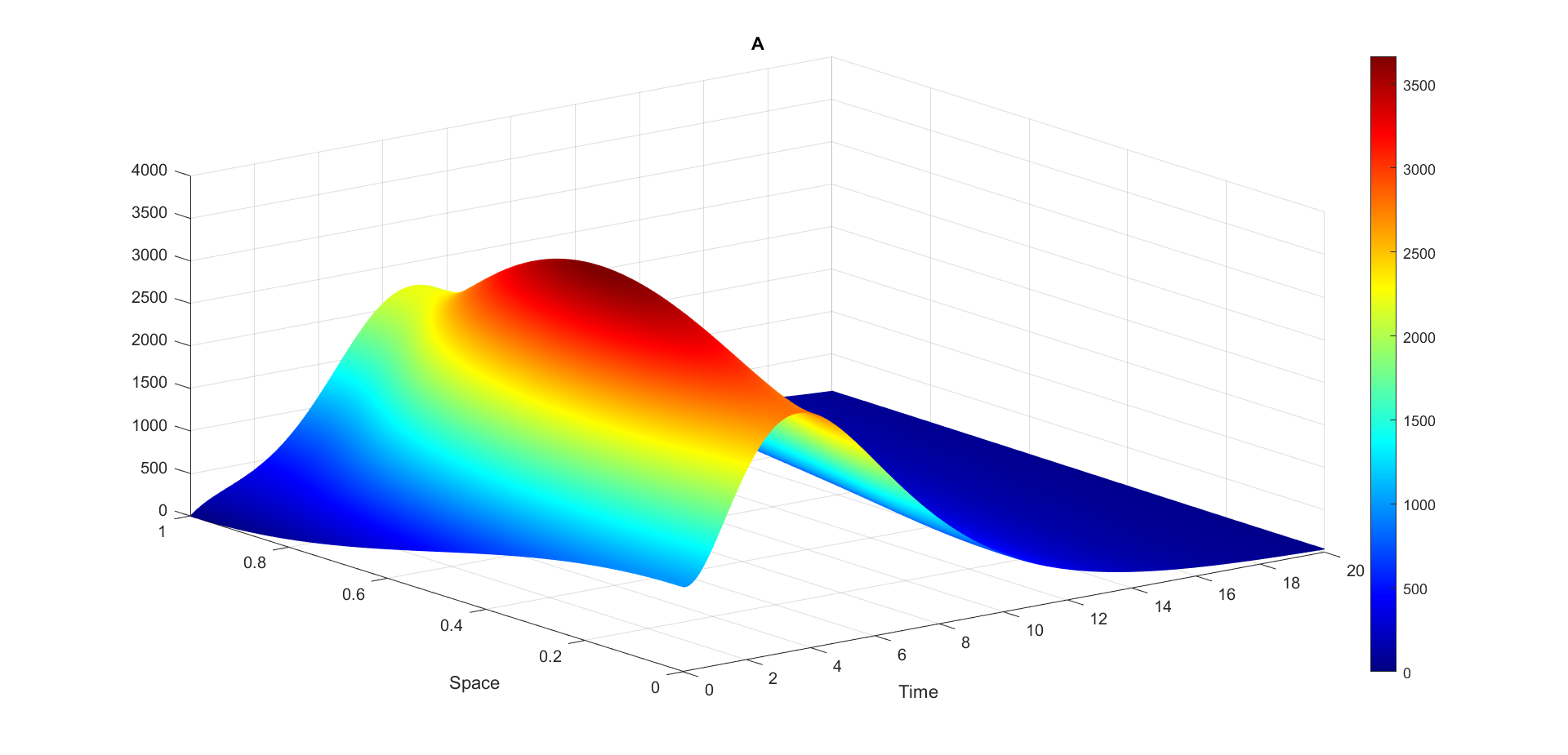}
		\put(62.3, 34.4){%
			\subfloat{%
				\begin{minipage}{0.5\textwidth}
					\includegraphics[width=0.5\textwidth]{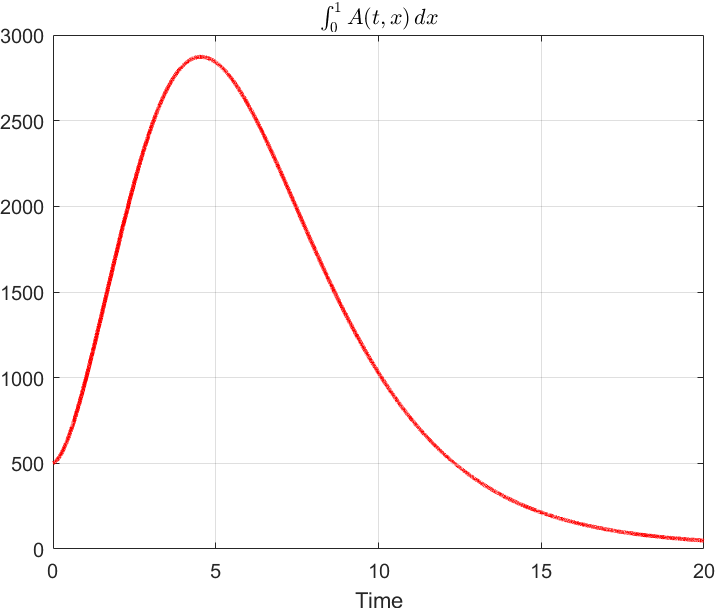}
		\end{minipage}}}
	\end{overpic}
	\caption{Changes in the Asymptomatic Group Without Control Actions.}
	\label{fig8}
\end{figure}
In contrast, when control measures are implemented, as illustrated in Figure \ref{fig9}, the number of asymptomatic individuals is significantly lower, not exceeding 1000 persons. This marked difference underscores the effectiveness of the control strategies in mitigating the spread of the virus. Despite these measures, some increase in asymptomatic cases is still observed, which is natural given that the controls do not directly target this group. The increase in asymptomatic individuals in the presence of controls can also be attributed to the initial pool of exposed individuals transitioning to the asymptomatic category. While the controls effectively reduce the rate at which new individuals become exposed and subsequently asymptomatic, they cannot entirely prevent the existing exposed individuals from progressing through the stages of infection. However, the overall impact is significantly mitigated, demonstrating the critical role of early and effective intervention in controlling the spread of the virus.
\begin{figure}[h!]
	\centering
	\begin{overpic}[width=\textwidth]{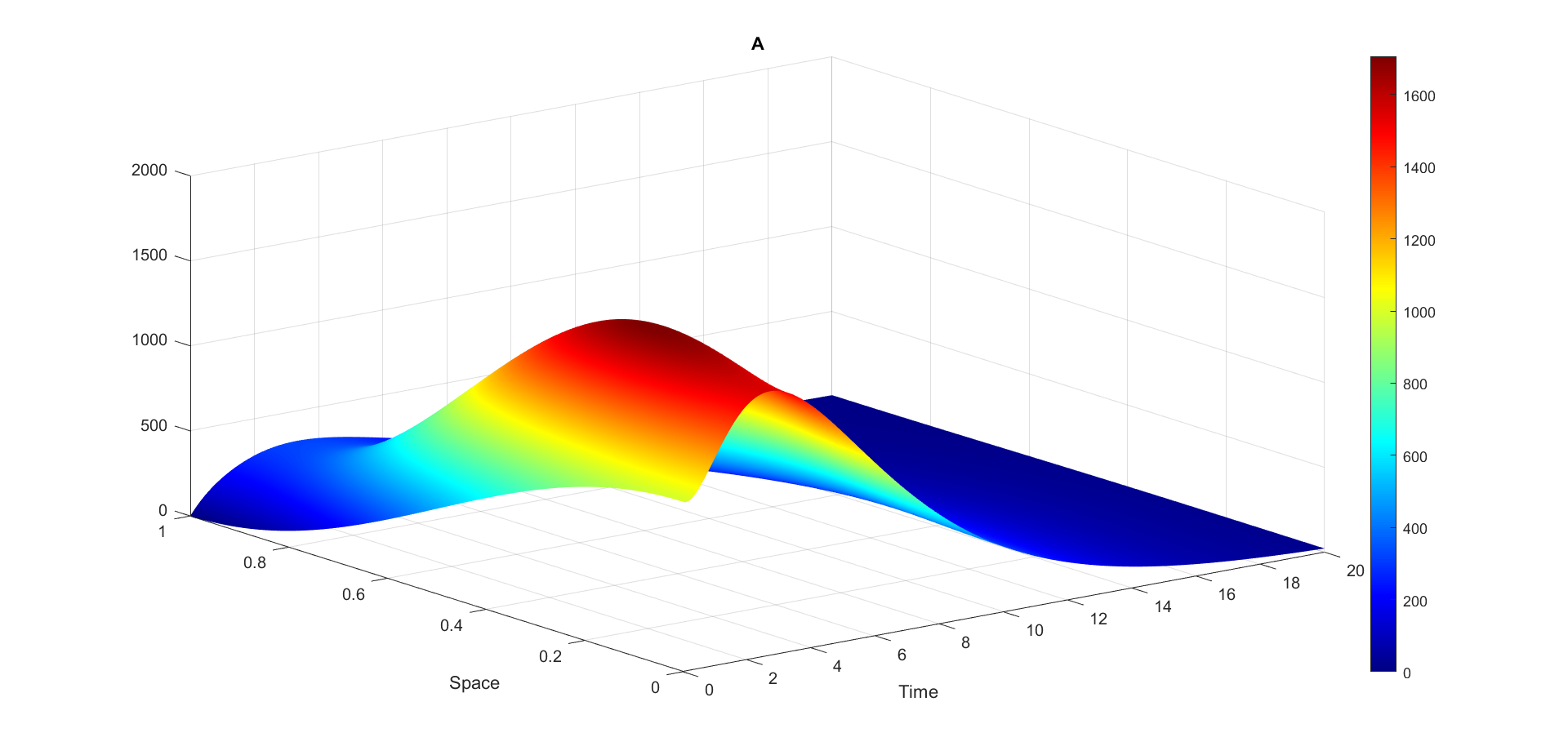}
		\put(62.3, 34.4){%
			\subfloat{%
				\begin{minipage}{0.5\textwidth}
					\includegraphics[width=0.5\textwidth]{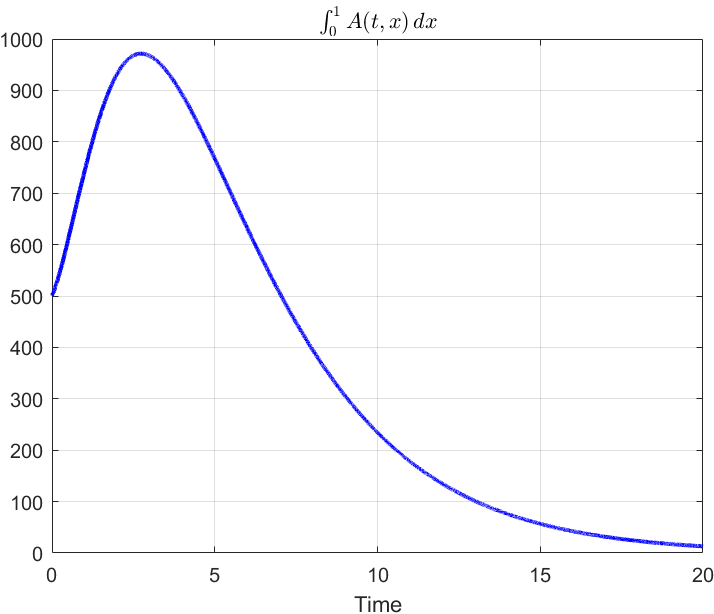}
				\end{minipage}
			}
		}
	\end{overpic}
	\caption{Changes in the Asymptomatic Group in the Presence of Controls.}
	\label{fig9}
\end{figure}
\begin{figure}[h!]
	\centering
	\begin{overpic}[width=\textwidth]{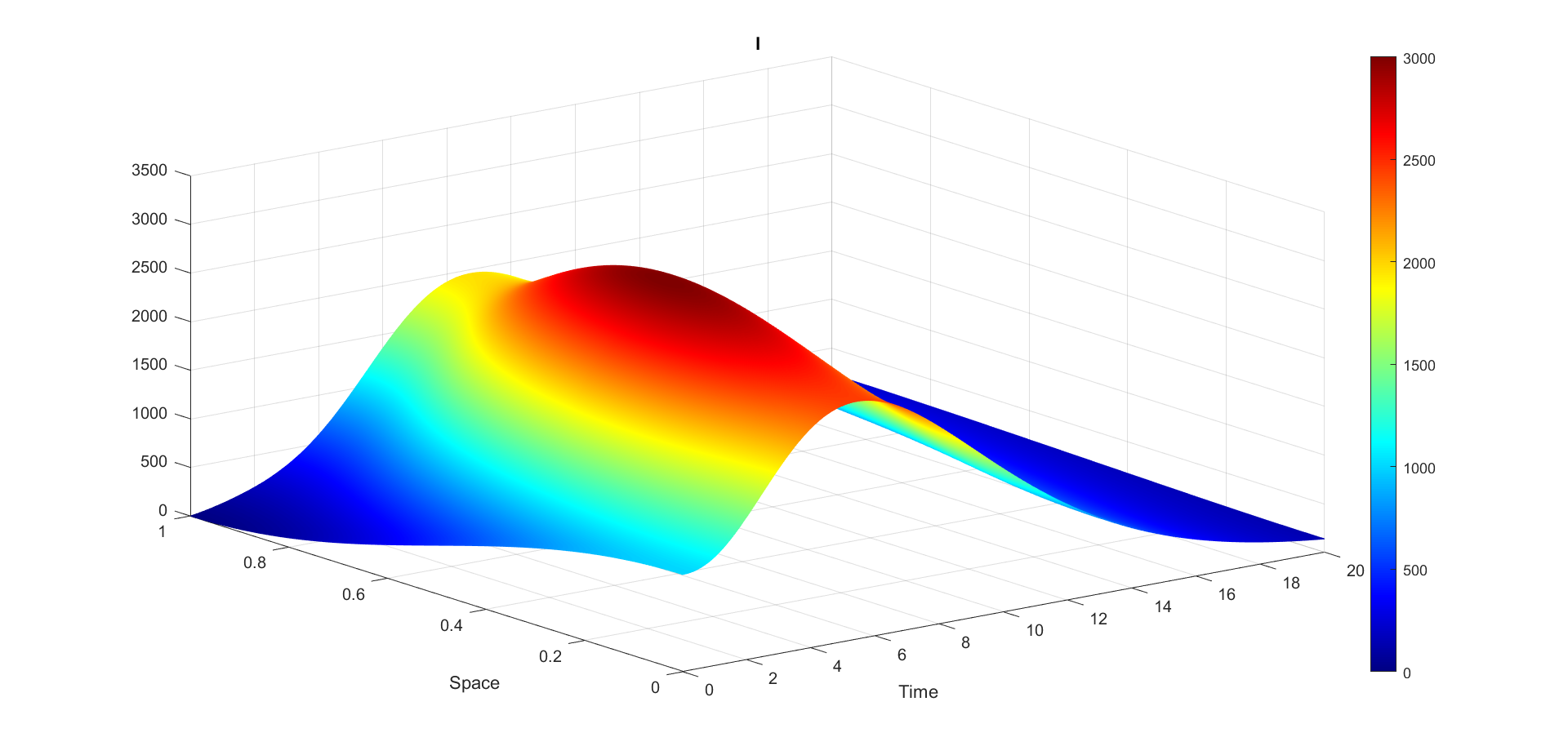}
		\vspace*{2cm}
		\put(62.3, 34.4){%
			\subfloat{%
				\begin{minipage}{0.5\textwidth}
					\includegraphics[width=0.5\textwidth]{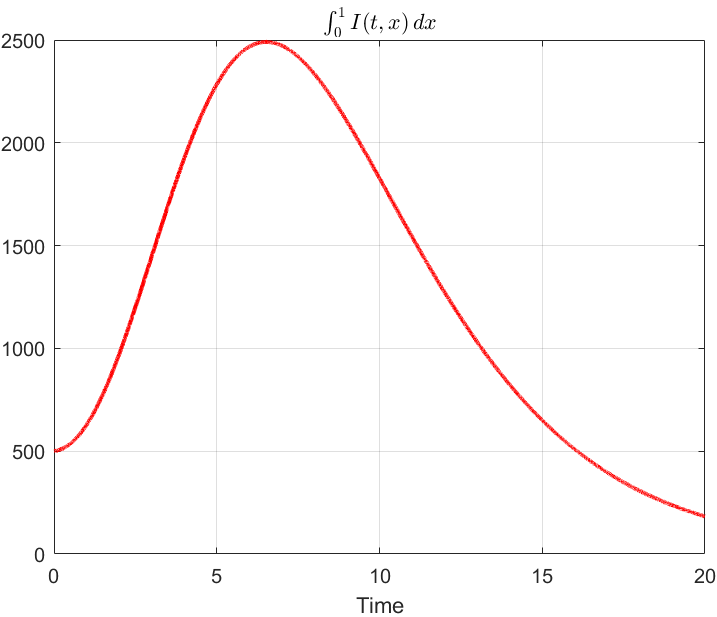}
				\end{minipage}
			}
		}
	\end{overpic}
	\caption{Changes in the Infected Group Without Control Actions.}
	\label{fig10}
\end{figure}
In Figure \ref{fig10}, we observe a substantial increase in the number of infected individuals in the absence of controls, with the infected population rising from 500 to over 2500 cases within approximately 15 days. This rapid escalation is due to the influx from both exposed and asymptomatic groups, significantly contributing to the overall number of active infections. The lack of control measures allows for unchecked transmission, enabling the virus to spread swiftly through the population. As the number of infected individuals grows, the healthcare access system becomes increasingly burdened, leading to higher morbidity and mortality rates.
\begin{figure}[h!]
	\centering
	\begin{overpic}[width=\textwidth]{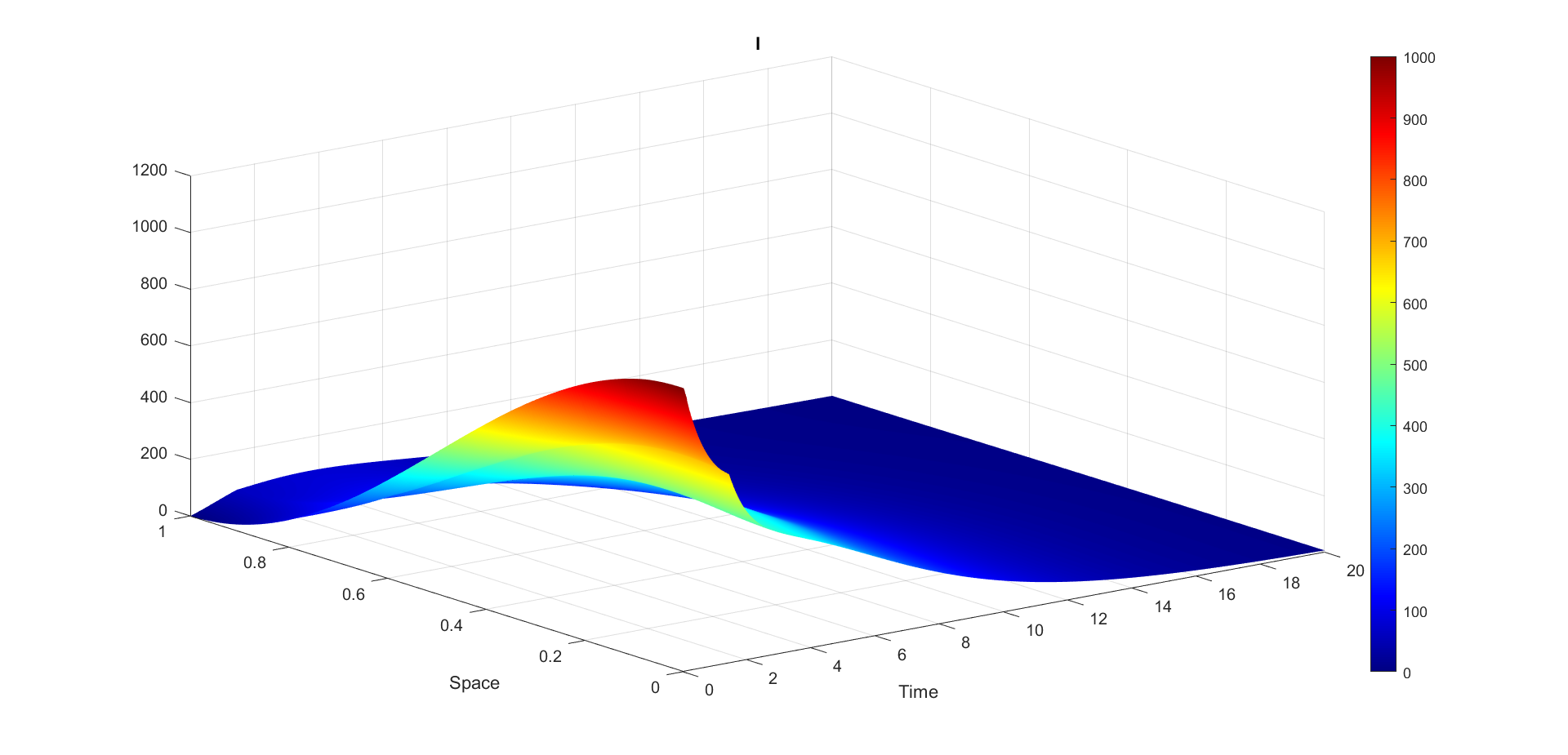}
		\vspace*{2cm}
		\put(62.3, 34.4){%
			\subfloat{%
				\begin{minipage}{0.5\textwidth}
					\includegraphics[width=0.5\textwidth]{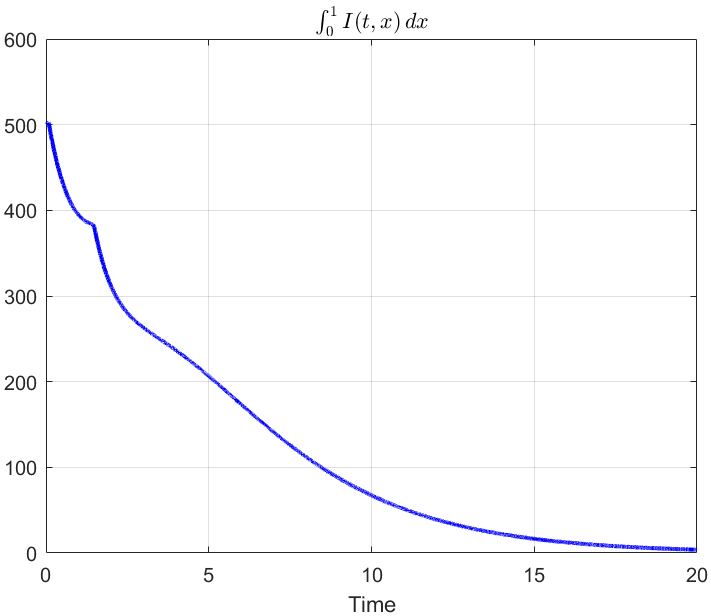}
				\end{minipage}
			}
		}
	\end{overpic}
	\caption{Changes in the Infected Group in the Presence of Controls.}
	\label{fig11}
\end{figure}
In contrast, when control measures are implemented, a markedly different trend is observed. As shown in Figure \ref{fig11}, the number of infected individuals starts decreasing from the first day of intervention, reaching zero within 20 days. This rapid reduction underscores the effectiveness of control measures in aiding recovery.
\begin{figure}[h!]
	\centering
	\includegraphics[scale=0.28]{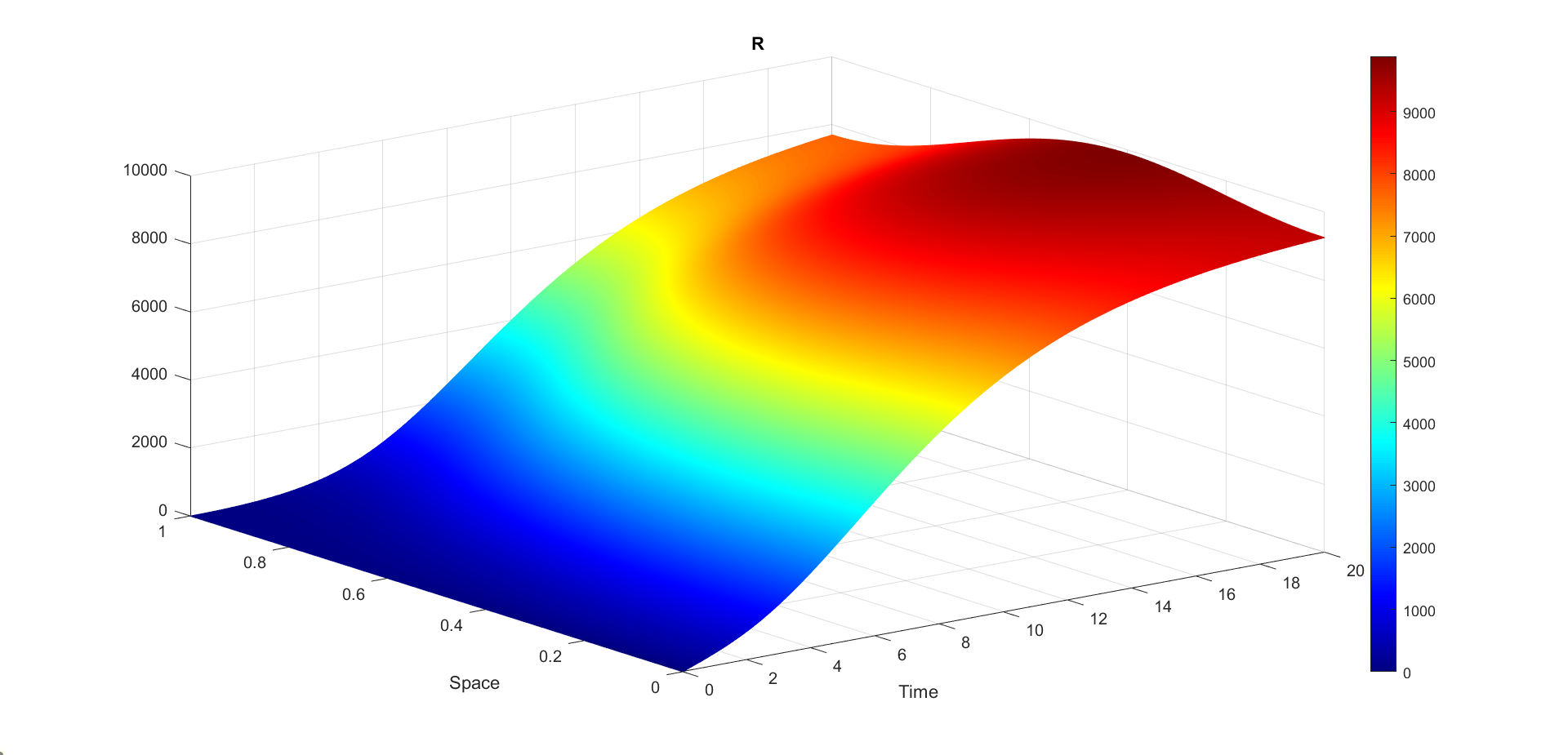}
	\caption{Changes in the Recovered Group Without Control Actions.}
	\label{fig12}
\end{figure}
\begin{figure}[h!]
	\centering
	\includegraphics[scale=0.28]{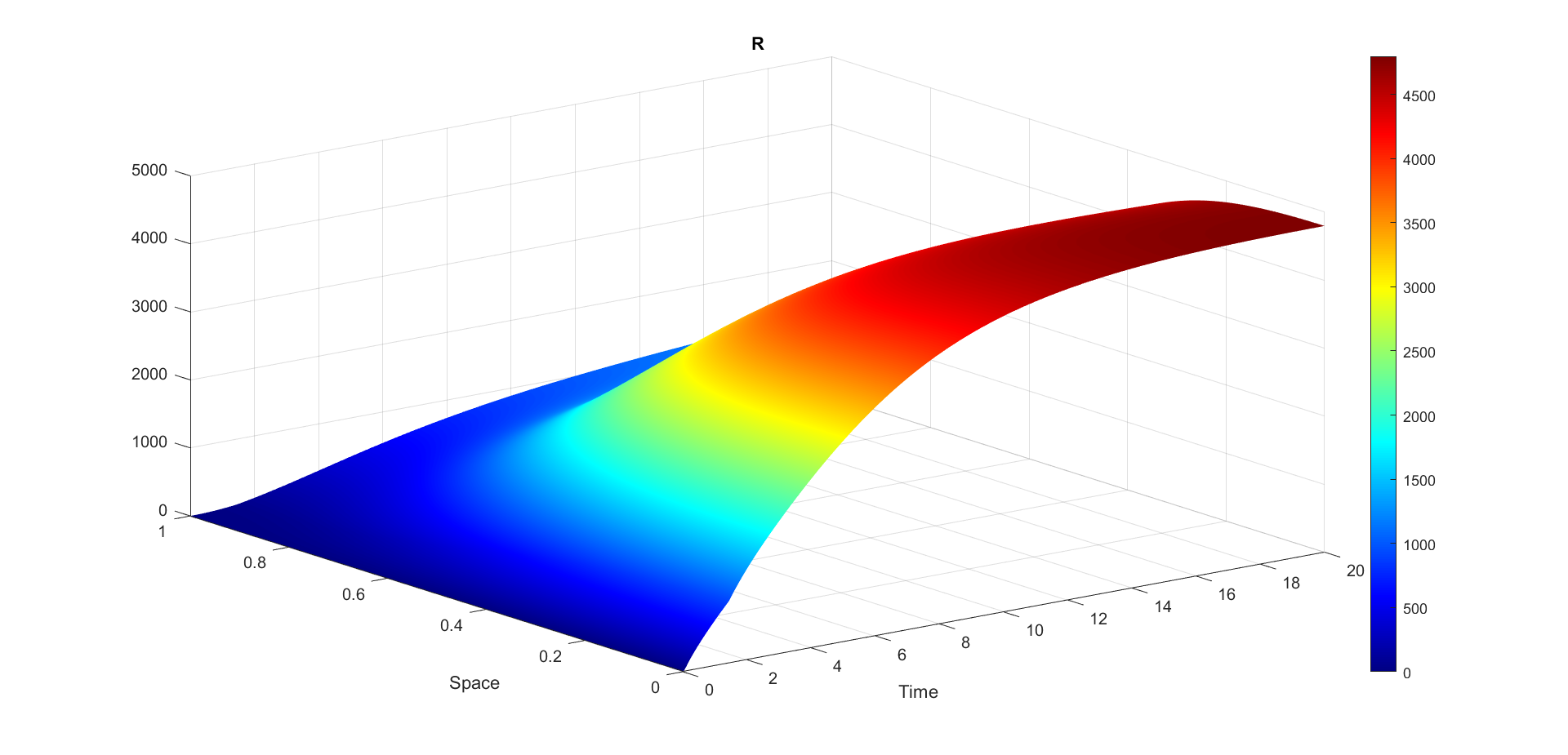}
	\caption{Changes in the Recovered Group in the Presence of Controls.}
	\label{fig13}
\end{figure}
Figure \ref{fig12} shows the evolution of the recovered population in the absence of control measures. The graph shows a significant number of individuals, more than 9100 people, recovering from the virus. This high number of recoveries suggests that almost the entire population will become infected at some point, as the virus spreads unchecked throughout the community. The large number of recoveries reflects the widespread transmission and eventual resolution of infections without intervention. In contrast, figure \ref{fig13} shows the effect of control measures on the recovery rate. With controls in place, the number of recoveries does not exceed 3000. This reduction is largely due to the implementation of quarantine measures, which effectively limit the spread of the virus and thus reduce the total number of infections. Figure \ref{fig5} shows that more than 4300 individuals are quarantined under these measures.

\begin{figure}[h!]
	\centering
	\includegraphics[scale=0.28]{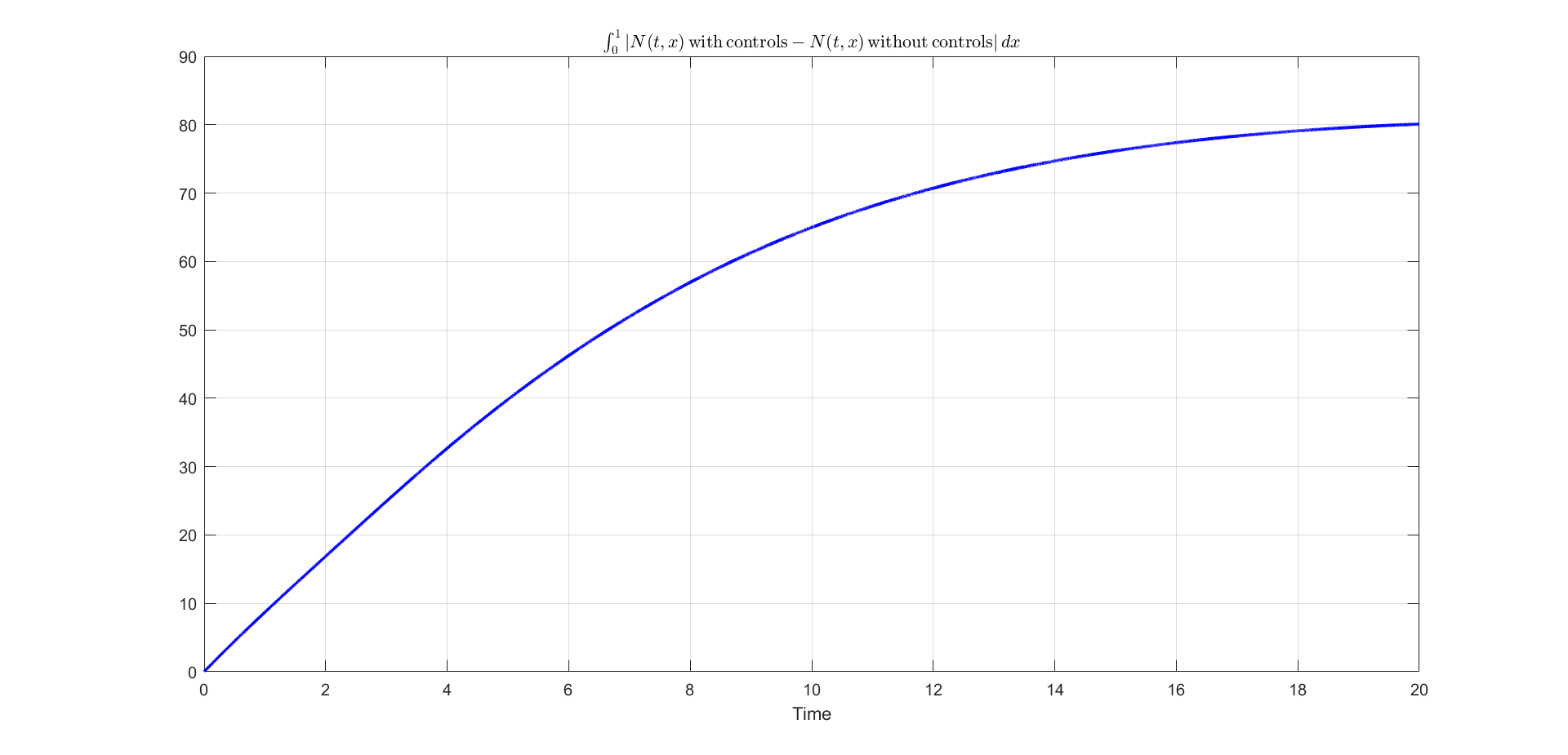}
	\caption{The absolute difference between the total populations of people with and without controls.}
	\label{fig14}
\end{figure}
\begin{figure}[h!]
	\centering
	\includegraphics[scale=0.28]{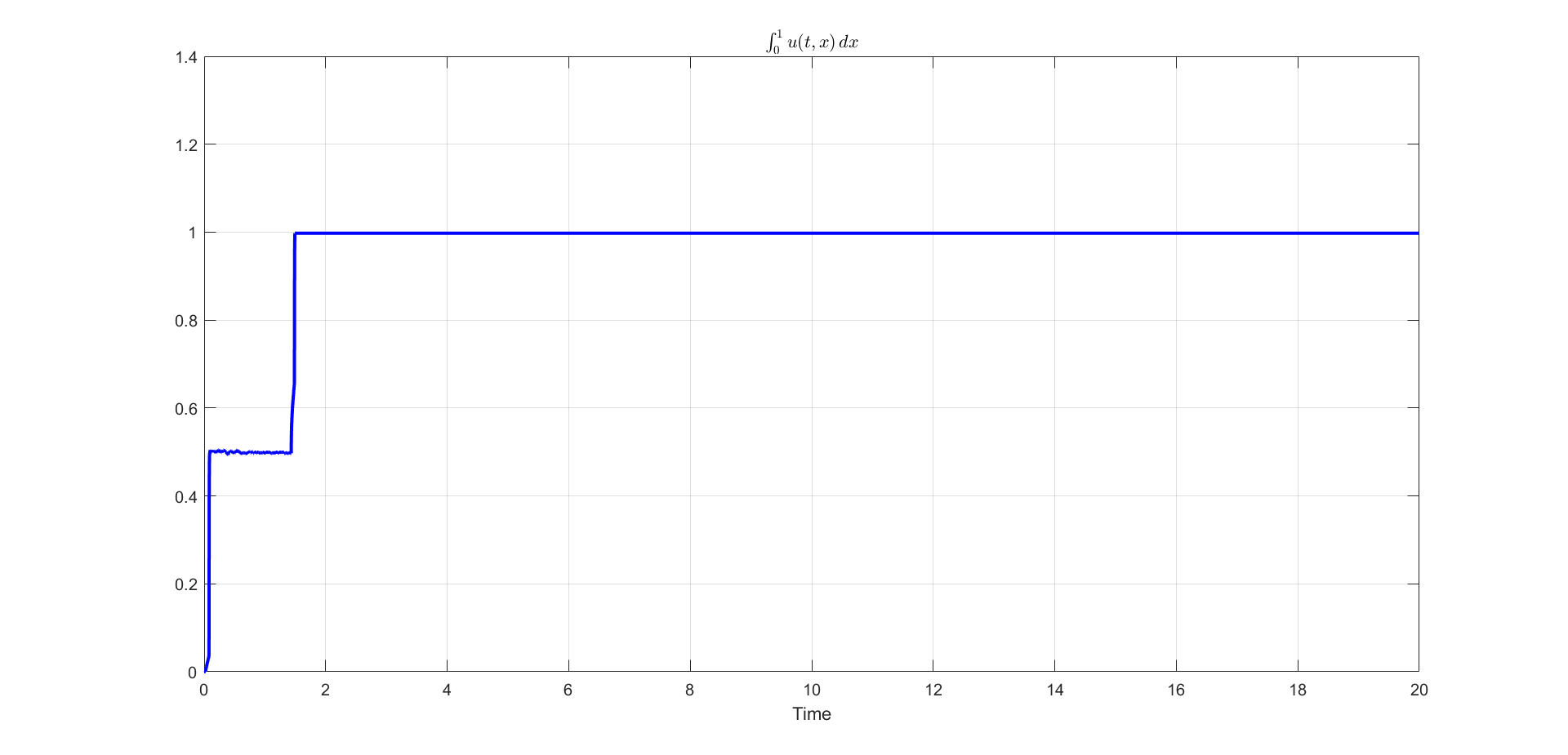}
	\caption{The evolution of treatment control function over time.}
	\label{fig15}
\end{figure}
\begin{figure}[h!]
	\centering
	\includegraphics[scale=0.28]{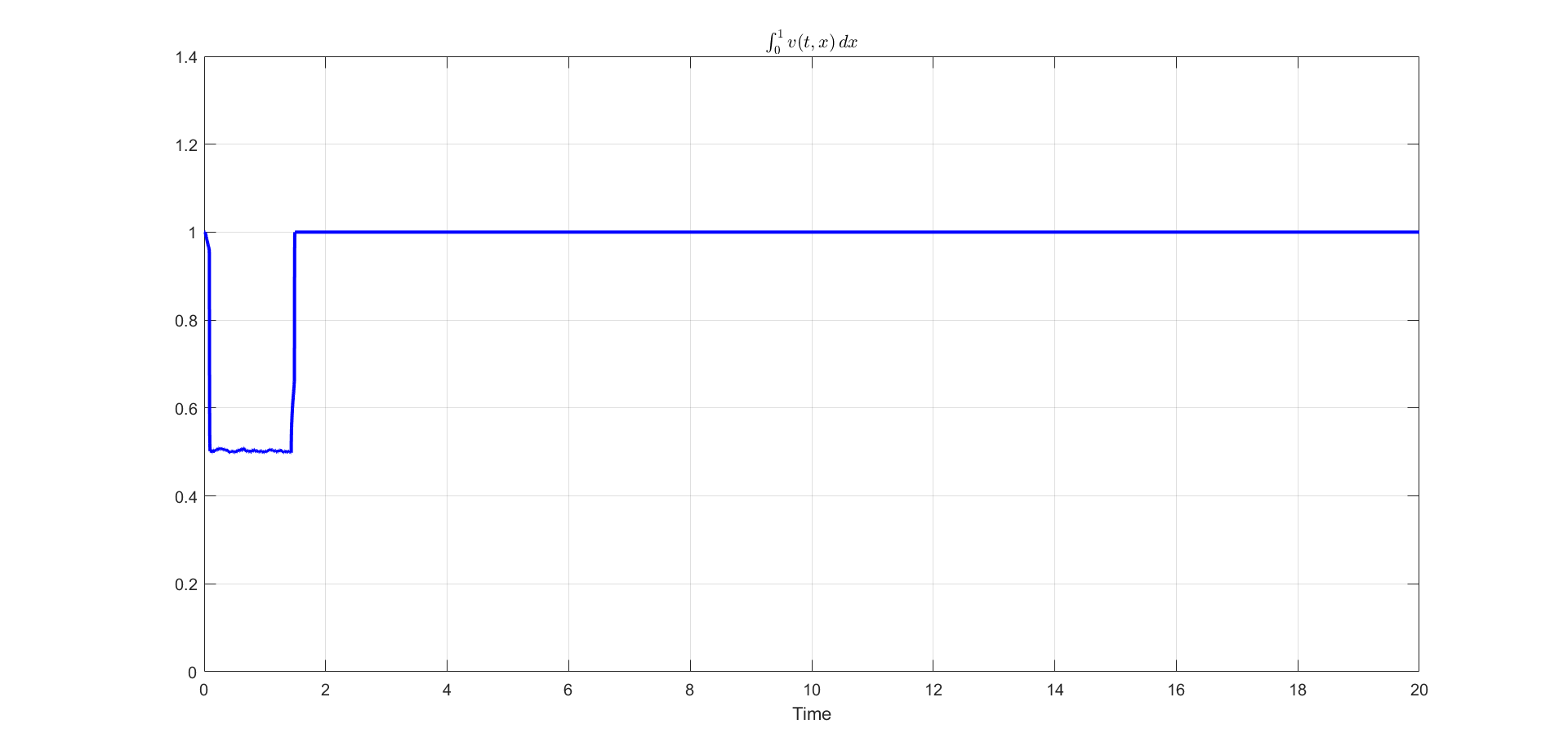}
	\caption{The evolution of quarantine control function over time.}
	\label{fig16}
\end{figure}
Figure \ref{fig14} shows the absolute total population of people with and without control measures. This demonstrates that control measures allow us to save more than 80 people from death due to the infection, which is particularly important for a population of no more than 8000 people.

\subsection{Discussion and Implications} \label{section 7}
\noindent

The analysis of the figures provides valuable insights into the dynamics of the COIVD-19 epidemic under different scenarios. In the absence of control measures, the susceptible population declines rapidly, reaching zero within a week, as illustrated in Figure \ref{fig2}. This results in a significant transition into the exposed category, leading to a surge in asymptomatic and infected cases. The introduction of control measures has a marked effect on the situation. Figure \ref{fig3} illustrates that the isolation of a proportion of the susceptible population effectively reduces the direct transition to the exposed group, thereby mitigating the overall impact. In the absence of controls, Figure \ref{fig6} illustrates a marked increase in the number of exposed individuals, reaching a peak of over 2500 within ten days. In contrast, the implementation of control measures (Figure \ref{fig7}) has resulted in a significantly lower peak of exposed individuals, with a maximum of 1000 cases. This indicates the effectiveness of indirect controls in limiting exposure. The growth of asymptomatic individuals is particularly concerning in the absence of control measures, as illustrated in Figure \ref{fig8}, with numbers reaching approximately 3000. This uncontrolled spread poses a significant risk due to asymptomatic transmission. Control measures, as depicted in Figure \ref{fig9}, significantly reduce this number to below 1000, demonstrating their efficacy in curbing the spread, despite the natural progression of existing exposed individuals. Those infected perceive a stark contrast between the two scenarios. In the absence of control measures, Figure \ref{fig10} illustrates a dramatic increase in the number of infected individuals, reaching over 2500 within 15 days and placing significant strain on healthcare Access systems. In contrast, the implementation of control measures results in a rapid decrease in infections, reaching zero within 20 days. This highlights the critical role of early intervention in the context of the pandemic. Furthermore, the recovery trends serve to reinforce the impact of the implemented controls. Figure \ref{fig12} illustrates a high recovery count of approximately 9100 individuals in the absence of controls, which reflects a widespread infection. However, the implementation of control measures, as illustrated in Figure \ref{fig13}, has resulted in a reduction in the number of recoveries to approximately 3000. This is attributed to the decreased infection rates, which have been supported by the quarantining of over 4300 individuals (Figure \ref{fig5}). The figures collectively demonstrate the profound effectiveness of control measures in managing and mitigating the spread of the COIVD-19 epidemic. Consequently, as the number of people who can be quarantined against infection increases and infected individuals are treated, more lives can be saved.
\section{Conclusion}
\noindent

In this paper, we have considered a spatiotemporal SEIAR epidemic model that includes a new group called the quarantined population (Q), referred to as the SQEIAR epidemic model. Our primary objective was to mitigate the spread of the infection by reducing the numbers of exposed, asymptomatic, and infected individuals, while implementing quarantine measures for the susceptible population through the application of optimal control theory. Our strategy employs two distinct control actions: the first involves imposing quarantine measures on susceptible individuals to prevent further spread of the disease; the second focuses on providing treatment to infected individuals to reduce the overall infection rate. The dynamics of our model are described by a set of Partial Differential Equations (PDEs). We have rigorously investigated the mathematical properties of our model, including the existence, boundedness, and positivity of the solutions.  Furthermore, we have established the existence of an optimal solution, characterized by an optimal control strategy, by minimizing a specifically defined cost functional. To validate our theoretical findings, we conducted numerical simulations, applying our model to the real-world scenario of the COVID-19 pandemic. These simulations demonstrate the practical viability and effectiveness of our approach, highlighting the potential for significant reductions in infection rates through the implementation of our proposed control strategies. 

\vspace{0.2cm}
\noindent
\textbf{Future Direction:} A key aspect of our model lies in the control action \( v(t,x) \in \left[ 0, \frac{1}{n} \right] \), which represents the quarantine applied to susceptible individuals. It is clear that as the number of regions \( n \) increases, the intensity of the quarantine measure in each region decreases. The decrease in \( v(t,x) \) as \( n \) increases can be interpreted as a shift from a more concentrated quarantine strategy (fewer regions, higher quarantine intensity) to a more distributed strategy (more regions, lower quarantine intensity), could be a way to minimize the economic and social impact of quarantine by distributing the control more evenly across the space while still attempting to control the spread of the disease. However, determining the optimal value of \( n \) remains crucial. This, of course, will depend on many factors such as population density, geographical characteristics, healthcare Access infrastructure, transmission rate, economic and social costs, and more. Future work could focus on investigating \emph{the ideal number} $n$ of regions needed to maximize the effectiveness of quarantine measures while minimizing the associated social, economic, and logistical costs. This optimization would balance disease transmission control with practical constraints, providing a more refined strategy for epidemic management.

\section*{Acknowledgments}
The authors are appreciative to the anonymous referees for carefully reading the work and providing insightful feedback.
\section*{Disclosure statement}
This work is free from any conflicts of interest.
\section*{Funding}
The authors state that they have no known financial interests or personal relationships that could have influenced the work presented in this paper.

\end{document}